\documentclass[12pt,graybox]{article}
\usepackage{mathptmx}
\usepackage[scaled=1.5]{helvet}
\usepackage{courier}
\usepackage[latin9]{inputenc}
\usepackage{textcomp}
\usepackage{mathrsfs}
\usepackage{enumitem}
\usepackage{amsmath}
\usepackage{amsthm}
\usepackage{amssymb}
\usepackage{graphicx}
\usepackage{esint}
\PassOptionsToPackage{normalem}{ulem}
\usepackage{ulem}
\usepackage[unicode=true,pdfusetitle,
 bookmarks=true,bookmarksnumbered=false,bookmarksopen=false,
 breaklinks=false,pdfborder={0 0 1},colorlinks=false]
 {hyperref}
\usepackage{breakurl}

\makeatletter
\theoremstyle{plain}
\newtheorem{thm}{\protect\theoremname}
  \theoremstyle{definition}
  \newtheorem{defn}[thm]{\protect\definitionname}
  \theoremstyle{plain}
  \newtheorem{prop}[thm]{\protect\propositionname}
  \theoremstyle{definition}
  \newtheorem{example}[thm]{\protect\examplename}
  \theoremstyle{remark}
  \newtheorem{rem}[thm]{\protect\remarkname}
  \theoremstyle{plain}
  \newtheorem{cor}[thm]{\protect\corollaryname}
  \theoremstyle{plain}
  \newtheorem{lem}[thm]{\protect\lemmaname}

\makeatother

  \providecommand{\corollaryname}{Corollary}
  \providecommand{\definitionname}{Definition}
  \providecommand{\examplename}{Example}
  \providecommand{\lemmaname}{Lemma}
  \providecommand{\propositionname}{Proposition}
  \providecommand{\remarkname}{Remark}
\providecommand{\theoremname}{Theorem}

\begin{document}

\title{\textbf{The Caccioppoli Ultrafunctions} }

\author{Vieri Benci, Luigi Carlo Berselli\thanks{corresponding author},
and Carlo Romano Grisanti}

\maketitle
Dipartimento di Matematica, Universit\`a degli Studi di Pisa, Via F.
Buonarroti 1/c, 56127 Pisa, ITALY 
\begin{abstract}
Ultrafunctions are a particular class of functions defined on a hyperreal
field $\mathbb{R}^{\ast}\supset\mathbb{R}$. They have been introduced
and studied in some previous works (\cite{ultra},\cite{belu2012},\cite{belu2013}).
In this paper we introduce a particular space of ultrafunctions which
has special properties, especially in term of localization of functions
together with their derivatives. An appropriate notion of integral
is then introduced which allows to extend in a consistent way the
integration by parts formula, the Gauss theorem and the notion of
perimeter. This new space we introduce, seems suitable for applications
to PDE's and Calculus of Variations. This fact will be illustrated
by a simple, but meaningful example.

\noindent \textbf{Keywor}\textbf{\normalsize{}d}\textbf{s}. Ultrafunctions,
Non Archimedean Mathematics, Non Standard Analysis, Delta function,
distributions. 
\end{abstract}
\tableofcontents{}

\section{Introduction\label{sec:Intro}}

The Caccioppoli ultrafunctions can be considered as a kind generalized
functions. In many circumstances, the notion of real function is not
sufficient to the needs of a theory and it is necessary to extend
it. Among people working in partial differential equations, the theory
of distributions of L. Schwartz is the most commonly used, but other
notions of generalized functions have been introduced by J.F. Colombeau
\cite{col85} and M. Sato \cite{sa59,sa60}. This paper deals with
a new kind of generalized functions, called ``ultrafunctions'', which
have been introduced recently in \cite{ultra} and developed in \cite{belu2012,belu2013,milano,algebra,beyond}.
They provide generalized solutions to certain equations which do not
have any solution, not even among the distributions. 

Actually, the ultrafunctions are pointwise defined on a subset of
$\left(\mathbb{R}^{\ast}\right)^{N},$ where $\mathbb{R}^{\ast}$
is the field of hyperreal numbes, namely the numerical field on which
nonstandard analysis (NSA in the sequel) is based. We refer to Keisler
\cite{keisler76} for a very clear exposition of NSA and in the following,
starred quantities are the natural extensions of the corresponding
classical quantities.

The main novelty of this paper is that we introduce the space of Caccioppoli
ultrafunctions $V_{\Lambda}(\Omega)$. They satisfy special properties
which are very powerful in applications to Partial Differential Equations
and Calculus of Variations. The construction of this space is rather
technical, but contains some relevant improvements with respect to
the previous notions present in the literature (see e.g. \cite{ultra,belu2012,belu2013,milano,algebra,beyond,BGG,nap}). 

The main peculiarities of the ultrafunctions in $V_{\Lambda}(\Omega)$
are the following: there exist a generalized partial derivative $D_{i}$
and a generalized integral $\sqint$ (called poinwise integral) such
that 
\begin{enumerate}
\item the generalized derivative is a local operator namely, if $supp(u)\subset E^{*}$
(where $E$ is an open set), then $supp(D_{i}u)\subset E^{*}$.
\item $\forall u,v\in V_{\Lambda}(\Omega)$, 
\begin{equation}
\sqint D_{i}uv\,dx=-\sqint uD_{i}v\,dx\,\,;\label{eq:bellona}
\end{equation}
\item the ``generalized'' Gauss theorem holds for any measurable set $A$
(see Theorem \ref{thm:41})
\[
\sqint_{A}D\cdot\phi\,dx=\sqint_{\partial A}\phi\cdot\mathbf{n}_{A}\,dS\,\,;
\]
\item to any distribution $T\in\mathcal{\mathscr{D}}^{\prime}\left(\Omega\right)$
we can associate an equivalence class of ultrafunctions $\left[u\right]$
such that, $\forall v\in\left[u\right],\,\forall\varphi\in\mathcal{\mathscr{D}}\left(\Omega\right),$
\[
st\left(\sqint v\varphi\text{\textdegree}\,dx\right)=\left\langle T,\varphi\right\rangle ,
\]
where $st(\cdot)$ denotes the standard part of an hyperreal number. 
\end{enumerate}
The most relevant point, which is not present in the previous approaches
to ultrafunctions, is that we are able the extend the notion of partial
derivative so that it is a local operator and it satisfies the usual
formula valid when integrating by parts, at the price of a suitable
extension of the integral as well. In the proof of this fact, the
Caccioppoli sets play a fundamental role.

It is interesting to compare the result about the Caccioppoli ultrafunctions
with the well-known Schwartz impossibility theorem:\emph{ ``there
does not exist a differential algebra $(\mathfrak{A},+,\otimes,D)$
in which the distributions can be embedded, where $D$ is a linear
operator that extends the distributional derivative and satisfies
the Leibniz rule (namely $D(u\otimes v)=Du\otimes v+u\otimes Dv$)
and $\otimes$ is an extension of the pointwise product on $\mathcal{C}(\mathbb{R})$.'' }

The ultrafunctions extend the space of distributions; they do not
violate the Schwartz theorem since the Leibniz rule, in general, does
not hold (see Remark \ref{rem:Schw}). Nevertheless, we can prove
the integration by parts rule (\ref{eq:bellona}) and the Gauss' divergence
theorem (with the appropriate extension $\sqint$ of the usual integral),
which are the main tools used in the applications. These results are
a development of the theory previously introduced in \cite{algebra}
and \cite{gauss}.

The theory of ultrafunctions makes deep use of the techniques of NSA
presented via the notion of $\Lambda$-limit. This presentation has
the advantage that a reader, which does not know NSA, is able to follow
most of the arguments. 

In the last section we present some very simple examples to show that
the ultrafunctions can be used to perform a precise mathematical analysis
of problems which are not tractable via the distributions.

\subsection{Plan of the paper}

In section \ref{sec:lt}, we present a summary of the theory of $\Lambda$-limits
and their role in the development of the ultrafunctions using nonstandard
methods, especially in the context of transferring as much as possible
the language of classical analysis. In Section \ref{sec:Ultrafunctions},
we define the notion of ultrafunctions, with emphasis on the $\textit{pointwise integral}$.
In Section \ref{sec:Differential-calculus}, we define the most relevant
notion, namely the generalized derivative, and its connections with
the pointwise integral, together with comparison with the classical
and distributional derivative. In Section \ref{sec:FSU}, we show
how to construct a space satisfying all the properties of the generalized
derivative and integrals. This Section is the most technical and can
be skipped in a first reading. Finally, in Section \ref{sec:One-simple-application},
we present a general result and two very simple variational problem.
In particular, the second problem is very elementary but without solutions
in the standard $H^{1}$-setting. Nevertheless it has a natural and
explicit candidate as solution. We show how this can be described
by means of the language of ultrafunctions. 

\subsection{Notations\label{sec:not}}

If $X$ is a set and $\Omega$\ is a subset of $\mathbb{R}^{N}$,
then 
\begin{itemize}
\item $\mathcal{P}\left(X\right)$ denotes the power set of $X$ and $\mathcal{P}_{fin}\left(X\right)$
denotes the family of finite subsets of $X;$ 
\item $\mathfrak{F}\left(X,Y\right)$ denotes the set of all functions from
$X$ to $Y$ and $\mathfrak{F}\left(\Omega\right)=\mathfrak{F}\left(\Omega,\mathbb{R}\right)$; 
\item $\mathscr{\mathcal{C}}\left(\Omega\right)$ denotes the set of continuous
functions defined on $\Omega;$ 
\item $\mathcal{C}^{k}\left(\Omega\right)$ denotes the set of functions
defined on $\Omega$ which have continuous derivatives up to the order
$k$;
\item $H^{k,p}\left(\Omega\right)$ denotes the usual Sobolev space of functions
defined on $\Omega$; 
\item if $W\left(\Omega\right)$ is any function space, then $W_{c}\left(\Omega\right)$
will denote de function space of functions in $W\left(\Omega\right)$
having compact support; 
\item $\mathcal{C}_{0}\left(\Omega\cup\Xi\right),\,\,\Xi\subseteq\partial\Omega,$
denotes the set of continuous functions in $\mathcal{C}\left(\Omega\cup\Xi\right)\ $
which vanish for $x\in\Xi$;
\item $\mathcal{D}\left(\Omega\right)$ denotes the set of the infinitely
differentiable functions with compact support defined on $\Omega;\ \mathcal{D}^{\prime}\left(\Omega\right)$
denotes the topological dual of $\mathcal{D}\left(\Omega\right)$,
namely the set of distributions on $\Omega;$ 
\item for any $\xi\in\left(\mathbb{R}^{N}\right)^{\ast},\rho\in\mathbb{R}^{\ast}$,
we set $\mathfrak{B}_{\rho}(\xi)=\left\{ x\in\left(\mathbb{R}^{N}\right)^{\ast}:\ |x-\xi|<\rho\right\} $; 
\item $\mathfrak{supp}(f)=(supp(f))^{*}$ where $supp$ is the usual notion
of support of a function or a distribution;
\item $\mathfrak{mon}(x)=\{y\in\left(\mathbb{R}^{N}\right)^{\ast}:x\sim y\}$
where $x\sim y$ means that $x-y$ is infinitesimal;
\item $\mathfrak{gal}(x)=\{y\in\left(\mathbb{R}^{N}\right)^{\ast}:x-y\text{ is finite}\}$
;
\item if $W$ is a generic function space, its topological dual will be
denoted by $W^{\prime}$ and the pairing by $\left\langle \cdot,\cdot\right\rangle _{W}$
\item we denote by $\chi_{X}$ the indicator (or characteristic) function
of $X$, namely
\[
\chi_{X}(x)=\begin{cases}
1 & if\,\,\,x\in X\\
0 & if\,\,\,x\notin X\,;
\end{cases}
\]
\item $|X|$ will denote the cardinality of $X$.
\end{itemize}

\section{$\Lambda$-theory\label{sec:lt}}

In this section we present the basic notions of Non Archimedean Mathematics
and of Nonstandard Analysis, following a method inspired by \cite{BDN2003}
(see also \cite{ultra} and \cite{belu2012}).

\subsection{Non Archimedean Fields\label{naf}}

Here, we recall the basic definitions and facts regarding non-Archimedean
fields. In the following, ${\mathbb{K}}$ will denote an ordered field.
We recall that such a field contains (a copy of) the rational numbers.
Its elements will be called numbers. 
\begin{defn}
Let $\mathbb{K}$ be an ordered field and $\xi\in\mathbb{K}$. We
say that: 
\end{defn}
\begin{itemize}
\item $\xi$ is infinitesimal if, for all positive $n\in\mathbb{N}$, $|\xi|<\frac{1}{n}$; 
\item $\xi$ is finite if there exists $n\in\mathbb{N}$ such as $|\xi|<n$; 
\item $\xi$ is infinite if, for all $n\in\mathbb{N}$, $|\xi|>n$ (equivalently,
if $\xi$ is not finite). 
\end{itemize}
\begin{defn}
An ordered field $\mathbb{K}$ is called Non-Archimedean if it contains
an infinitesimal $\xi\neq0$. 
\end{defn}
It is easily seen that all infinitesimal are finite, that the inverse
of an infinite number is a nonzero infinitesimal number, and that
the inverse of a nonzero infinitesimal number is infinite. 
\begin{defn}
A superreal field is an ordered field $\mathbb{K}$ that properly
extends $\mathbb{R}$. 
\end{defn}
It is easy to show, due to the completeness of $\mathbb{R}$, that
there are nonzero infinitesimal numbers and infinite numbers in any
superreal field. Infinitesimal numbers can be used to formalize a
new notion of closeness: 
\begin{defn}
\label{def infinite closeness} We say that two numbers $\xi,\zeta\in{\mathbb{K}}$
are infinitely close if $\xi-\zeta$ is infinitesimal. In this case,
we write $\xi\sim\zeta$. 
\end{defn}
Clearly, the relation $\sim$ of infinite closeness is an equivalence
relation and we have the following theorem
\begin{thm}
If $\mathbb{K}$ is a superreal field, every finite number $\xi\in\mathbb{K}$
is infinitely close to a unique real number $r\sim\xi$, called the
\textbf{standard part} of $\xi$. 
\end{thm}
Given a finite number $\xi$, we denote its standard part by $st(\xi)$,
and we put $st(\xi)=\pm\infty$ if $\xi\in\mathbb{K}$ is a positive
(negative) infinite number. 
\begin{defn}
Let $\mathbb{K}$ be a superreal field, and $\xi\in\mathbb{K}$ a
number. The monad of $\xi$ is the set of all numbers that are infinitely
close to it: 
\[
\mathfrak{m}\mathfrak{o}\mathfrak{n}(\xi)=\{\zeta\in\mathbb{K}:\xi\sim\zeta\},
\]
and the galaxy of $\xi$ is the set of all numbers that are finitely
close to it: 
\[
\mathfrak{gal}(\xi)=\{\zeta\in\mathbb{K}:\xi-\zeta\ \text{is\ finite}\}
\]
\end{defn}
By definition, it follows that the set of infinitesimal numbers is
$\mathfrak{mon}(0)$ and that the set of finite numbers is $\mathfrak{gal}(0)$.

\subsection{The $\Lambda$-limit\label{subsec:Lambda-limit}}

In this section we introduce a particular non-Archimedean field by
means of $\Lambda$-theory \footnote{Readers expert in nonstandard analysis will recognize that $\Lambda$-theory
is equivalent to the superstructure constructions of Keisler (see
\cite{keisler76} for a presentation of the original constructions
of Keisler). } (for complete proofs and further informations the reader is referred
to \cite{benci99}, \cite{ultra} and \cite{belu2012}). To recall
the basics of $\Lambda$-theory we have to recall the notion of superstructure
on a set (see also \cite{keisler76}): 
\begin{defn}
Let $E$ be an infinite set. The superstructure on $E$ is the set
\[
V_{\infty}(E)=\bigcup_{n\in\mathbb{N}}V_{n}(E),
\]
where the sets $V_{n}(E)$ are defined by induction setting 
\[
V_{0}(E)=E
\]
and, for every $n\in\mathbb{N}$, 
\[
V_{n+1}(E)=V_{n}(E)\cup\mathcal{P}\left(V_{n}(E)\right).
\]

Here $\mathcal{P}\left(E\right)$ denotes the power set of $E.$ Identifying
the couples with the Kuratowski pairs and the functions and the relations
with their graphs, it follows that $V_{\infty}(E)$ contains almost
every usual mathematical object that can be constructed starting with
$E;$ in particular, $V_{\infty}(\mathbb{R})$, which is the superstructure
that we will consider in the following, contains almost every usual
mathematical object of analysis.
\end{defn}
Throughout this paper we let 
\[
\mathfrak{L}=\mathcal{P}_{fin}(V_{\infty}(\mathbb{R}))
\]
and we order $\mathfrak{L}$ via inclusion. Notice that $(\mathfrak{L},\subseteq)$
is a directed set. We add to $\mathfrak{L}$ a \emph{point at infinity}
$\Lambda\notin\mathfrak{L}$, and we define the following family of
neighborhoods of $\Lambda:$ 
\[
\{\{\Lambda\}\cup Q\mid Q\in\mathcal{U}\},
\]
where $\mathcal{U}$ is a \emph{fine ultrafilter} on $\mathfrak{L}$,
namely a filter such that 
\begin{itemize}
\item for every $A,B\subseteq\mathfrak{L}$, if $A\cup B=\mathfrak{L}$
then $A\in\mathcal{U}$ or $B\in\mathcal{U}$; 
\item for every $\lambda\in\mathfrak{L}$ the set $Q(\lambda):=\{\mu\in\mathfrak{L}\mid\lambda\subseteq\mu\}\in\mathcal{U}$. 
\end{itemize}
In particular, we will refer to the elements of $\mathcal{U}$ as
qualified sets and we will write $\Lambda=\Lambda(\mathcal{U})$ when
we want to highlight the choice of the ultrafilter. A function $\varphi:\,\mathfrak{L}\rightarrow E$
will be called \emph{net} (with values in E). If $\varphi(\lambda)$
is a real net, we have that

\[
\lim_{\lambda\rightarrow\Lambda}\varphi(\lambda)=L
\]
if and only if\medskip{}
\[
\forall\varepsilon>0,\,\exists Q\text{\ensuremath{\in}}\mathcal{U},\,\,such\,\,that\,\,\forall\lambda\text{\ensuremath{\in}}Q,\,|\varphi(\lambda)-L|<\varepsilon.
\]

\medskip{}

As usual, if a property $P(\lambda)$ is satisfied by any $\lambda$
in a neighborhood of $\Lambda$, we will say that it is \textbf{\emph{eventually}}
satisfied.

Notice that the $\Lambda$-topology satisfies these interesting properties:
\begin{prop}
\label{nino}If the net $\varphi(\lambda)$ takes values in a compact
set $K$, then it is a converging net. 
\end{prop}
\begin{proof}
Suppose that the net $\varphi(\lambda)$ has a subnet converging to
$L\in\mathbb{R}$. We fix $\varepsilon>0$ arbitrarily and we have
to prove that $Q_{\varepsilon}\in\mathcal{U}$ where
\[
Q_{\varepsilon}=\left\{ \lambda\in\mathfrak{L}\ |\ \left\vert \varphi(\lambda)-L\right\vert <\varepsilon\right\} .
\]
We argue indirectly and we assume that 
\[
Q_{\varepsilon}\notin\mathcal{U}
\]
 Then, by the definition of ultrafilter, $N=\mathfrak{L}\backslash Q_{\varepsilon}\in\mathcal{U}$
and hence
\[
\forall\lambda\in N,\ \left\vert \varphi(\lambda)-L\right\vert \geq\varepsilon.
\]
This contradicts the fact that $\varphi(\lambda)$ has a subnet which
converges to $L.$
\end{proof}
\begin{prop}
\label{prop:Ass}Assume that $\varphi:\,\mathfrak{L}\rightarrow E,$
where $E$ is a first countable topological space; then if
\end{prop}
\[
\lim_{\lambda\rightarrow\Lambda}\varphi(\lambda)=x_{0},
\]
\emph{there exists a sequence $\left\{ \lambda_{n}\right\} $ in}
$\mathfrak{L}$\emph{ such that
\[
\lim_{n\rightarrow\infty}\varphi(\lambda_{n})=x_{0}
\]
We refer to the sequence $\varphi_{n}:=\varphi(\lambda_{n})$ as a
subnet of $\varphi(\lambda)$.}
\begin{proof}
It follows easily from the definitions.
\end{proof}
\begin{example}
Let $\varphi:\mathfrak{\,L}\rightarrow V$ be a net with value in
bounded set of a reflexive Banach space equipped with the weak topology;
then 
\[
v:=\lim_{\lambda\rightarrow\Lambda}\varphi(\lambda),
\]
is uniquely defined and there exists a sequence $n\mapsto\varphi(\lambda_{n})$
which converges to $v$. 
\end{example}
\bigskip{}

\begin{defn}
The set\emph{ }of the hyperreal numbers $\mathbb{R}^{\ast}\supset\mathbb{R}$
is a set equipped with a topology $\tau$ such that
\begin{itemize}
\item every net $\varphi:\,\mathfrak{L}\rightarrow\mathbb{R}$ has a unique
limit in $\mathbb{R}^{\ast}$, if $\mathfrak{L}$ and $\mathbb{R}^{\ast}$
are equipped with the $\Lambda$ and the $\tau$ topology respectively;
\item $\mathbb{R}^{\ast}$ is the closure of $\mathbb{R}$ with respect
to the topology $\tau$;
\item $\tau$ is the coarsest topology which satisfies the first property. 
\end{itemize}
\end{defn}
\medskip{}

\textit{\emph{The existence of such}} $\mathbb{R}^{\ast}$ \textit{\emph{is
a well known fact in NSA. The limit }}\emph{$\xi\in\mathbb{R}^{\ast}$}\textit{\emph{
of a net}} $\varphi:\,\mathfrak{L}\rightarrow\mathbb{R}$\textit{\emph{
with respect to the}} $\tau$\textit{\emph{ topology, following \cite{ultra},
is called the}}\emph{ $\Lambda$-limit }\textit{\emph{of}}\emph{\ $\varphi$}\textit{\emph{
and the following notation will be used:}}
\begin{equation}
\xi=\lim_{\lambda\uparrow\Lambda}\varphi(\lambda)\label{eq:lambdauno}
\end{equation}
namely, we shall use the up-arrow ``$\uparrow$'' to remind that
the target space is equipped with the topology $\tau$.

Given 
\[
\xi:=\lim_{\lambda\uparrow\Lambda}\varphi(\lambda)\,\,and\,\,\eta:=\lim_{\lambda\uparrow\Lambda}\psi(\lambda),
\]
we set
\begin{equation}
\xi+\eta:=\lim_{\lambda\uparrow\Lambda}\left(\varphi(\lambda)+\psi(\lambda)\right),\label{eq:su}
\end{equation}
and 
\begin{equation}
\xi\cdot\eta:=\lim_{\lambda\uparrow\Lambda}\left(\varphi(\lambda)\cdot\psi(\lambda)\right).\label{eq:pr}
\end{equation}

Then the following well known theorem holds:
\begin{thm}
The definitions (\ref{eq:su}) and (\ref{eq:pr}) are well posed and
$\mathbb{R}^{*}$, equipped with these operations, is a non-Archimedean
field.
\end{thm}
\begin{rem}
We observe that the field of hyperreal numbers is defined as a sort
of completion of the real numbers. In fact $\mathbb{R}^{*}$ is isomorphic
to the ultrapower
\[
\mathbb{R}\mathfrak{^{L}/\mathfrak{I}}
\]
where
\[
\mathfrak{I}=\left\{ \varphi:\mathfrak{L}\rightarrow\mathbb{R}\,|\,\varphi(\lambda)=0\,\,\,eventually\right\} 
\]
The isomorphism resembles the classical one between the real numbers
and the equivalence classes of Cauchy sequences. This method is well
known for the construction of real numbers starting from rationals. 
\end{rem}

\subsection{Natural extension of sets and functions}

For our purposes it is very important that the notion of $\Lambda$-limit
can be extended to sets and functions (but also to differential and
integral operators) in order to have a much wider set of objects to
deal with, to enlarge the notion of variational problem and of variational
solution. 

So we will define the $\Lambda$-limit of any bounded net of mathematical
objects in $V_{\infty}(\mathbb{R})$ (a net $\varphi:\,\mathfrak{L}\rightarrow V_{\infty}(\mathbb{R})$
is called bounded if there exists $n\in\mathbb{N}$ such that, $\forall\lambda\in\mathcal{\mathfrak{L}},\varphi(\lambda)\in V_{n}(\mathbb{R})$).
To do this, let us consider a net
\begin{equation}
\varphi:\mathcal{\mathfrak{L}}\rightarrow V_{n}(\mathbb{R}).\label{net}
\end{equation}
We will define $\lim_{\lambda\uparrow\Lambda}\varphi(\lambda)$ by
induction on $n$. 
\begin{defn}
For $n=0,$ $\lim_{\lambda\uparrow\Lambda}\varphi(\lambda)$ is defined
by (\ref{eq:lambdauno}). By induction we may assume that the limit
is defined for $n-1$ and we define it for the net (\ref{net}) as
follows:
\[
\lim_{\lambda\uparrow\mathcal{\textrm{\ensuremath{\Lambda}}}}\varphi(\lambda)=\left\{ \lim_{\lambda\uparrow\mathcal{\textrm{\ensuremath{\Lambda}}}}\psi(\lambda)\ |\ \psi:\mathfrak{\mathcal{\mathfrak{L}}}\rightarrow V_{n-1}(\mathbb{R}),\ \forall\lambda\in\mathcal{\mathfrak{L}},\ \psi(\lambda)\in\varphi(\lambda)\right\} .
\]

A mathematical entity (number, set, function or relation) which is
the $\Lambda$-limit of a net is called \textbf{internal}. 
\end{defn}
\bigskip{}

\begin{defn}
If $\forall\lambda\in\mathfrak{L,}$ $E_{\lambda}=E\in V_{\infty}(\mathbb{R}),$
we set $\lim_{\lambda\uparrow\Lambda}\ E_{\lambda}=E^{\ast},\ $namely
\[
E^{\ast}:=\left\{ \lim_{\lambda\uparrow\Lambda}\psi(\lambda)\ |\ \psi(\lambda)\in E\right\} .
\]
$E^{\ast}$ is called the \textbf{natural extension }of $E.$ 
\end{defn}
Notice that, while the $\Lambda$-limit\ of a constant sequence of
numbers gives this number itself, a constant sequence of sets gives
a larger set, namely $E^{\ast}$. In general, the inclusion $E\subseteq E^{\ast}$
is proper.

Given any set $E,$ we can associate to it two sets: its natural extension
$E^{\ast}$ and the set $E^{\sigma},$ where
\begin{equation}
E^{\sigma}=\left\{ X^{\ast}\ |\ X\in E\right\} .\label{sigmaS}
\end{equation}

Clearly $E^{\sigma}$ is a copy of $E$, however it might be different
as set since, in general, $X^{\ast}\neq X.$ 
\begin{rem}
If $\varphi:\mathfrak{\,L}\rightarrow X$ is a net with values in
a topological space we have the usual limit 
\[
\lim_{\lambda\rightarrow\Lambda}\varphi(\lambda),
\]
which, by Proposition \ref{nino}, always exists in the Alexandrov
compactification $X\cup\left\{ \infty\right\} $. Moreover we have
that the $\Lambda$-limit always exists and it is an element of $X^{*}$.
In addition, the $\Lambda$-limit of a net is in $X^{\sigma}$ if
and only if $\varphi$ is eventually constant. If $X=\mathbb{R}$
and both limits exist, then 

\begin{equation}
\lim_{\lambda\rightarrow\Lambda}\varphi(\lambda)=st\left(\lim_{\lambda\uparrow\Lambda}\varphi(\lambda)\right).\label{eq:teresa}
\end{equation}
\end{rem}
The above equation suggests the following definition.
\begin{defn}
\label{def:St}If $X$ is a topological space equipped with a Hausdorff
topology, and $\xi\in X^{*}$ we set
\[
St_{X}\left(\xi\right)=\lim_{\lambda\rightarrow\Lambda}\varphi(\lambda),
\]
if there is a net $\varphi:\mathfrak{\,L}\rightarrow X$ converging
in the topology of $X$ and such that
\[
\xi=\lim_{\lambda\uparrow\Lambda}\varphi(\lambda),
\]
and
\[
St_{X}\left(\xi\right)=\infty
\]
otherwise.
\end{defn}
\medskip{}

By the above definition we have that 
\[
\lim_{\lambda\rightarrow\Lambda}\varphi(\lambda)=St_{X}\left(\lim_{\lambda\uparrow\Lambda}\varphi(\lambda)\right).
\]

\begin{defn}
Let 
\[
f_{\lambda}:\ E_{\lambda}\rightarrow\mathbb{R},\ \ \lambda\in\mathfrak{L},
\]
be a net of functions. We define a function
\[
f:\left(\lim_{\lambda\uparrow\Lambda}\ E_{\lambda}\right)\rightarrow\mathbb{R}^{\ast}
\]
as follows: for every $\xi\in\left(\lim_{\lambda\uparrow\Lambda}\ E_{\lambda}\right)$
we set
\[
f\left(\xi\right):=\lim_{\lambda\uparrow\Lambda}\ f_{\lambda}\left(\psi(\lambda)\right),
\]
where $\psi(\lambda)$ is a net of numbers such that 
\[
\psi(\lambda)\in E_{\lambda}\ \ \text{and}\ \ \lim_{\lambda\uparrow\Lambda}\psi(\lambda)=\xi.
\]
A function which is a $\Lambda$-\textit{limit\ is called }\textbf{\textit{internal}}\textit{.}
In particular if, $\forall\lambda\in\mathfrak{L,}$ 
\[
f_{\lambda}=f,\ \ \ \ f:\ E\rightarrow\mathbb{R},
\]
we set 
\[
f^{\ast}=\lim_{\lambda\uparrow\Lambda}\ f_{\lambda}.
\]
$f^{\ast}:E^{\ast}\rightarrow\mathbb{R}^{\ast}$ is called the \textbf{natural
extension }of $f.$ If we identify $f$ with its graph, then $f^{\ast}$
is the graph of its natural extension.
\end{defn}

\subsection{Hyperfinite sets and hyperfinite sums\label{HE}}
\begin{defn}
An internal set is called \textbf{hyperfinite} if it is the $\Lambda$-limit
of a net $\varphi:\mathfrak{L}\rightarrow\mathfrak{F}$ where $\mathfrak{F}$
is a family of finite sets. 
\end{defn}
For example, if $E\in V_{\infty}(\mathbb{R})$, the set
\[
\widetilde{E}=\lim_{\lambda\uparrow\Lambda}\left(\lambda\cap E\right)
\]
is hyperfinite. Notice that
\[
E^{\sigma}\subset\widetilde{E}\subset E^{*}
\]
so, we can say that every set is contained in a hyperfinite set. 

\medskip{}

It is possible to add the elements of an hyperfinite set of numbers
(or vectors) as follows: let
\[
A:=\ \lim_{\lambda\uparrow\Lambda}A_{\lambda},
\]
be an hyperfinite set of numbers (or vectors); then the hyperfinite
sum of the elements of $A$ is defined in the following way: 
\[
\sum_{a\in A}a=\ \lim_{\lambda\uparrow\Lambda}\sum_{a\in A_{\lambda}}a.
\]
In particular, if $A_{\lambda}=\left\{ a_{1}(\lambda),...,a_{\beta(\lambda)}(\lambda)\right\} \ $with\ $\beta(\lambda)\in\mathbb{N},\ $then
setting 
\[
\beta=\ \lim_{\lambda\uparrow\Lambda}\ \beta(\lambda)\in\mathbb{N}^{\ast},
\]
we use the notation
\[
\sum_{j=1}^{\beta}a_{j}=\ \lim_{\lambda\uparrow\Lambda}\sum_{j=1}^{\beta(\lambda)}a_{j}(\lambda).
\]

\section{Ultrafunctions\label{sec:Ultrafunctions}}

\subsection{Caccioppoli spaces of ultrafunctions\label{subsec:Cacc}}

Let $\Omega$ be an open bounded set in $\mathbb{R}^{N}$, and let
$W(\Omega)$ be a (real or complex) vector space such that $\mathcal{D}(\overline{\Omega})\subseteq W(\Omega)\subseteq L^{1}(\Omega).$
\begin{defn}
\label{def:ultraf-1} A space of ultrafunctions modeled over the space
$W(\Omega)$ is given by
\[
W_{\Lambda}(\Omega):=\lim_{\lambda\uparrow\Lambda}W_{\lambda}(\Omega)=\left\{ \lim_{\lambda\uparrow\Lambda}f_{\lambda}\,\,|\,\,f_{\lambda}\in W_{\lambda}(\Omega)\right\} ,
\]
where $W_{\lambda}(\Omega)\subset W(\Omega)$ is an increasing net
of finite dimensional spaces such that
\[
W_{\lambda}(\Omega)\supseteq Span(W(\Omega)\cap\lambda).
\]
\end{defn}
\medskip{}

So, given any vector space of functions $W(\Omega)$, the space of
ultrafunction generated by $\left\{ W_{\lambda}(\Omega)\right\} $
is a vector space of hyperfinite dimension that includes $W(\Omega)^{\sigma}$,
as well as other functions in $W(\Omega)^{*}$. Hence the ultrafunctions
are particular internal functions
\[
u:\ \overline{\Omega}^{*}\rightarrow\mathbb{R}^{\ast}.
\]

\begin{defn}
Given a space of ultrafunctions $W_{\Lambda}(\Omega)$, a $\sigma$-basis
is an internal set of ultrafunctions $\left\{ \sigma_{a}(x)\right\} _{a\in\Gamma}$
such that, $\Omega\subset\Gamma\subset\Omega^{*}$ and $\forall u\in W_{\Lambda}(\Omega)$,
we can write
\[
u(x)=\sum_{a\in\Gamma}u(a)\sigma_{a}(x).
\]
It is possible to prove (see e.g. \cite{ultra}) that every space
of ultrafunctions has a $\sigma$-basis. Clearly, if $a,b\in\Gamma$
then $\sigma_{a}(b)=\delta_{ab}$ where $\delta_{ab}$ denotes the
Kronecker delta. 
\end{defn}
Now we will introduce a class of spaces of ultrafunctions suitable
for most applications. To do this, we need to recall the notion of
Caccioppoli set:
\begin{defn}
A Caccioppoli set $E$ is a Borel set such that $\chi_{E}\in BV,$
namely such that $\nabla(\chi_{E})$ (the distributional gradient
of the characteristic function of $E$) is a finite Radon measure
concentrated on $\partial E$. 
\end{defn}
The number
\[
p(E):=\left\langle |\nabla(\chi_{E})|,\,1\right\rangle 
\]
is called Caccioppoli perimeter of $E.$ From now on, with some abuse
of notation, the above expression will be written as follows:
\[
\int|\nabla(\chi_{E})|dx;
\]
this expression makes sense since ``$|\nabla(\chi_{E})|dx$'' is
a measure.

If $E\subset\overline{\Omega}$ is a measurable set, we define the
\emph{density function} of $E$ as follows:
\begin{equation}
\theta_{E}(x)=st\left(\frac{m(B_{\eta}(x)\cap E^{*})}{m(B_{\eta}(x)\cap\left(\overline{\Omega}\right)^{*})}\right),\label{eq:estate}
\end{equation}
where $\eta$ is a fixed infinitesimal and $m$ is the Lebesgue measure.

Clearly $\theta_{E}(x)$ is a function whose value is 1 in $int(E)$
and 0 in $\mathbb{R}^{N}\setminus\overline{E}$; moreover, it is easy
to prove that $\theta_{E}(x)$ is a measurable function and we have
that
\[
\int\theta_{E}(x)dx=m(E);
\]
also, if $E$ is a bounded Caccioppoli set,
\[
\int|\nabla\theta_{E}|dx=p(E).
\]

\begin{defn}
\emph{\label{eq:grisa}A set $E$ is called special Caccioppoli set
if it is open, bounded and $m\left(\partial E\right)=0.$ The family
of special Caccioppoli sets will be denoted by $\mathfrak{C}(\mathrm{\Omega})$.}
\end{defn}
Now we can define a space $V(\Omega)$ suitable for our aims: 
\begin{defn}
\label{def:AB}\emph{A function $f\in V(\Omega)$ if and only if
\[
f(x)=\sum_{k=1}^{n}f_{k}(x)\theta_{E_{k}}(x)
\]
where $f_{k}\in\mathcal{\mathscr{C}}(\mathbb{R}^{N})$, $E_{k}\in\mathfrak{C(\textrm{\ensuremath{\Omega}})}$,
and $n$ is a number which depends on $f$. Such a function, will
be called }\textbf{\emph{Caccioppoli function}}\emph{.}
\end{defn}
Notice that $V(\Omega)$ is a module over the ring $\mathscr{C}(\overline{\Omega})$
and that, $\forall f\in V(\Omega)$,
\begin{equation}
\left(\int|f(x)|\,dx=0\right)\Rightarrow\left(\forall x\in\mathbb{R}^{N},\ f(x)=0\right).\label{eq:ola}
\end{equation}
Hence, in particular, 
\[
\left(\int|f(x)|^{2}dx=0\right)^{1/2},
\]
is a norm (and not a seminorm).
\begin{defn}
\label{def:scorre} $V_{\Lambda}(\Omega)$ is called \emph{Caccioppoli}
\emph{space of ultrafunctions} if it satisfies the following properties:

\begin{enumerate}[label=(\roman*)]
\item $V_{\Lambda}(\Omega)$ is modeled on the space $V(\Omega)$ ;
\item $V_{\Lambda}(\Omega)$ has a $\sigma$-basis $\left\{ \sigma_{a}(x)\right\} _{a\in\Gamma}$,
$\Gamma\subset\left(\mathbb{R}^{N}\right)^{*}$, such that $\forall a\in\Gamma$
the support of $\sigma_{a}$ is contained in $\mathfrak{mon}(a)$.
\end{enumerate}
\end{defn}
Notice 

The existence of a Caccioppoli space of ultrafunctions will be proved
in Section \ref{sec:FSU}.
\begin{rem}
Usually in the study of PDE's, the function space where to work depends
on the problem or equation which we want to study. The same fact is
true in the world of ultrafunctions. However, the Caccioppoli space
$V_{\Lambda}(\Omega)$ have a special position since it satisfies
the properties required by a large class of problems. First of all
$V_{\Lambda}(\Omega)\subset\left(L^{1}(\Omega)\right)^{*}$. This
fact allows to define the pointwise integral (see next sub-section)
for all the ultrafunctions. This integral turns out to be a very good
tool. However, the space $L^{1}$ is not a good space for modeling
ultrafunctions, since they are defined pointwise while the functions
in $L^{1}$ are defined a.e. Thus, we are lead to the space $L^{1}(\Omega)\cap\mathcal{C}(\overline{\Omega})$,
but this space does not contain functions such as $f(x)\theta_{E}(x)$
which are important in many situations; for example the Gauss' divergence
theorem can be formulated as follows
\[
\int\nabla\cdot F(x)\theta_{E}(x)dx=\int_{\partial E}\mathbf{n}\cdot F(x)dS\,\,,
\]
whenever the vector field $F$ and $E$ are sufficiently smooth. Thus
the space $V_{\Lambda}(\Omega)$ seems to be the right space for a
large class of problems.
\end{rem}

\subsection{The pointwise integral}

From now on we will denote by $V_{\Lambda}(\Omega)$ a fixed Caccioppoli
space of ultrafunctions and by $\left\{ \sigma_{a}(x)\right\} _{a\in\Gamma}$
a fixed $\sigma$-basis as in Definition \ref{def:scorre}. If $u\in V_{\Lambda}(\Omega)$,
we have that
\begin{equation}
\int^{*}u(x)dx=\sum_{a\in\Gamma}u(a)\eta_{a}\,\,,\label{eq:fighissima}
\end{equation}
where
\[
\eta_{a}:=\int^{*}\sigma_{a}(x)\,dx.
\]

The equality (\ref{eq:fighissima}) suggests the following definition:
\begin{defn}
\label{def:sqint-1}For any internal function $g\,:\,\Omega{}^{*}\rightarrow\mathfrak{\mathbb{R}}^{*}$,
we set 
\[
\sqint g(x)dx:=\sum_{q\in\Gamma}g(q)\eta_{q\,}.
\]
In the sequel we will refer to $\sqint$ as to the \emph{pointwise
integral}. 
\end{defn}
From Definition \ref{def:sqint-1}, we have that 
\begin{equation}
\forall u\in V_{\Lambda}(\Omega),\ \int^{*}u(x)dx=\sqint u(x)dx,\label{eq:lilla-2}
\end{equation}
and in particular,

\begin{equation}
\forall f\in V(\Omega),\ \int f(x)dx=\sqint f^{*}(x)dx.\label{eq:lilla-3}
\end{equation}
But in general these equalities are not true for $L^{1}$ functions.
For example if
\[
f(x)=\begin{cases}
1 & if\ x=x_{0}\in\Omega,\\
0 & if\ x\neq x_{0\,,}
\end{cases}
\]
we have that $\int^{*}f^{*}(x)dx=\int f(x)dx=0$, while $\sqint f^{*}(x)dx=\eta_{x_{0}}>0$.
However, for any set $E\in\mathfrak{C(\mathrm{\Omega})}$ and any
function $f\in\mathcal{\mathscr{C}}(\overline{\Omega})$
\[
\sqint f^{*}(x)\theta_{E}(x)\,dx=\int_{E}f(x)\,dx,
\]
in fact
\[
\sqint f^{*}(x)\theta_{E}(x)\,dx=\int^{*}f^{*}(x)\theta_{E}(x)\,dx=\int f(x)\theta_{E}(x)\,dx=\int_{E}f(x)\,dx.
\]
Then, if $f(x)\geq0$ and $E$ is a bounded open set, we have that
\[
\sqint f^{*}(x)\chi_{E}dx<\sqint f^{*}(x)\theta_{E}(x)\,dx<\sqint f^{*}(x)\chi_{\overline{E}}dx.
\]
since
\[
\chi_{E}<\theta_{E}<\chi_{\overline{E}}.
\]

As we will see in the following part of this paper, in many cases,
it is more convenient to work with the pointwise integral $\sqint$
rather than with the natural extension of the Lebesgue integral $\int^{*}$. 
\begin{example}
If $\partial E$ is smooth, we have that $\forall x\in\partial E,\,\theta_{E}(x)=\frac{1}{2}$
and hence, if $E$ is open,
\begin{align*}
\sqint f^{*}(x)\chi_{E}(x)\,dx & =\sqint f^{*}(x)\theta_{E}(x)\,dx-\frac{1}{2}\sqint f^{*}(x)\chi_{\partial E}(x)\,dx\\
 & =\int_{E}f(x)\,dx-\frac{1}{2}\sqint f^{*}(x)\chi_{\partial E}(x)\,dx,
\end{align*}
and similarly
\[
\sqint f^{*}(x)\chi_{\overline{E}}(x)\,dx=\int_{E}f(x)\,dx+\frac{1}{2}\sqint f^{*}(x)\chi_{\partial E}(x)\,dx;
\]
of course, the term $\frac{1}{2}\sqint f^{*}(x)\chi_{\partial E}(x)\,dx$
is an infinitesimal number and it is relevant only in some particular
problems.
\end{example}
\medskip{}

The pointwise integral allows us to define the following scalar product:
\begin{equation}
\sqint u(x)v(x)dx=\sum_{q\in\Gamma}u(q)v(q)\eta_{q}.\label{eq:rina}
\end{equation}

From now on, the norm of an ultrafunction will be given by 
\[
\left\Vert u\right\Vert =\left(\sqint|u(x)|^{2}\ dx\right)^{\frac{1}{2}}.
\]
Notice that 
\[
\sqint u(x)v(x)dx=\int^{*}u(x)v(x)dx\Leftrightarrow uv\in V_{\Lambda}(\Omega).
\]

\begin{thm}
\label{thm:sette}If $\left\{ \sigma_{a}(x)\right\} _{a\in\Gamma}$
is a $\sigma$\textup{-}\textup{\emph{basis}}, then
\[
\left\{ \frac{\sigma_{a}(x)}{\sqrt{\eta_{a}}}\right\} _{a\in\Gamma}
\]
is a orthonormal basis with respect to the scalar product (\ref{eq:rina}).
Hence for every $u\in V_{\Lambda}(\Omega),$
\begin{equation}
u(x)=\sum_{q\in\Gamma}\frac{1}{\eta_{q}}\left(\sqint u(\xi)\sigma_{q}(\xi)d\xi\right)\sigma_{q}(x).\label{eq:lella}
\end{equation}
Moreover, we have that\textup{\emph{ }}\textup{
\begin{equation}
\forall a\in\Gamma,\ \left\Vert \sigma_{a}\right\Vert ^{2}=\eta_{a}.\label{eq:lilla}
\end{equation}
}
\end{thm}
\begin{proof}
By (\ref{eq:rina}), we have that
\[
\sqint\sigma_{a}(x)\sigma_{b}(x)dx=\sum_{q\in\Gamma}\sigma_{a}(q)\sigma_{b}(q)\eta_{q}=\sum_{q\in\Gamma}\delta_{aq}\delta_{bq}\eta_{q}=\delta_{ab}\eta_{a},
\]
then the result. By the above equality, taking $b=a$ we get (\ref{eq:lilla}). 
\end{proof}

\subsection{The $\delta$-bases}

Next, we will define the \textit{delta ultrafunctions}:
\begin{defn}
\label{dede}Given a point $q\in\Omega^{\ast},$ we denote by $\delta_{q}(x)$
an ultrafunction in $V_{\Lambda}(\Omega)$ such that 
\begin{equation}
\forall v\in V_{\Lambda}(\Omega),\ \sqint v(x)\delta_{q}(x)\ dx=v(q),\label{deltafunction}
\end{equation}
and $\delta_{q}(x)$ is called \emph{delta (or the Dirac) ultrafunction}
concentrated in $q$. 
\end{defn}
Let us see the main properties of the delta ultrafunctions:
\begin{thm}
\label{delta} The delta ultrafunction satisfies the following properties:

\begin{enumerate}
\item For every $q\in\overline{\Omega}^{\ast}$ there exists an unique delta
ultrafunction concentrated in $q;$
\item for every $a,\ b\in\overline{\Omega}^{*},\ \delta_{a}(b)=\delta_{b}(a);$
\item $\left\Vert \delta_{q}\right\Vert ^{2}=\delta_{q}(q).$ 
\end{enumerate}
\end{thm}
\begin{proof}
1. Let $\left\{ e_{j}\right\} _{j=1}^{\beta}$ be an orthonormal real
basis of $V_{\Lambda}(\Omega),$ and set 
\begin{equation}
\delta_{q}(x)=\sum_{j=1}^{\beta}e_{j}(q)e_{j}(x).\label{eq:deltaserie}
\end{equation}
Let us prove that $\delta_{q}(x)$ actually satisfies (\ref{deltafunction}).
Let $v(x)=\sum_{j=1}^{\beta}v_{j}e_{j}(x)$ be any ultrafunction.
Then
\begin{eqnarray*}
\sqint v(x)\delta_{q}(x)dx & = & \sqint\left(\sum_{j=1}^{\beta}v_{j}e_{j}(x)\right)\left(\sum_{k=1}^{\beta}e_{k}(q)e_{k}(x)\right)dx=\\
 & = & \sum_{j=1}^{\beta}\sum_{k=1}^{\beta}v_{j}e_{k}(q)\sqint e_{j}(x)e_{k}(x)dx=\\
 & = & \sum_{j=1}^{\beta}\sum_{k=1}^{\beta}v_{j}e_{k}(q)\delta_{jk}=\sum_{j=1}^{\beta}v_{k}e_{k}(q)=v(q).
\end{eqnarray*}
So $\delta_{q}(x)$ is a delta ultrafunction concentrated in $q$.
It is unique: infact, if $\gamma_{q}(x)$ is another delta ultrafunction
concentrated in $q$, then for every $y\in\overline{\Omega}^{*}$
we have:
\[
\delta_{q}(y)-\gamma_{q}(y)=\sqint(\delta_{q}(x)-\gamma_{q}(x))\delta_{y}(x)dx=\delta_{y}(q)-\delta_{y}(q)=0,
\]
and hence $\delta_{q}(y)=\gamma_{q}(y)$ for every $y\in\overline{\Omega}^{\ast}.$

2.$\ \delta_{a}\left(b\right)=\sqint\delta_{a}(x)\delta_{b}(x)\ dx=\delta_{b}\left(a\right).$

3. $\left\Vert \delta_{q}\right\Vert ^{2}=\sqint\delta_{q}(x)\delta_{q}(x)\,dx=\delta_{q}(q)$.\\
\end{proof}
By the definition of $\Gamma$, $\forall a,b\in\Gamma$, we have that
\begin{equation}
\sqint\delta_{a}(x)\sigma_{b}(x)dx=\sigma_{a}(b)=\delta_{ab}.\label{mimma}
\end{equation}
From this it follows readily the following result 
\begin{prop}
The set $\left\{ \delta_{a}(x)\right\} _{a\in\Gamma}$ $(\Gamma\subset\Omega^{\ast})$
is the dual basis of the sigma-basis; it will be called the $\delta$-basis
of $V_{\Lambda}(\Omega)$. 
\end{prop}
\medskip{}

Let us examine the main properties of the $\delta$-basis:
\begin{prop}
The $\delta$-basis, satisfies the following properties:

\begin{enumerate}[label=(\roman*)]
\item $u(x)=\sum_{q\in\Gamma}\left[\sqint\sigma_{q}(\xi)u(\xi)d\xi\right]\delta_{q}(x);$
\item \textup{\emph{$\forall a,b\in\Gamma,\ \sigma_{a}(x)=\eta_{a}\delta_{a}(x);$}}
\item \textup{$\forall a\in\Gamma,\ \left\Vert \delta_{a}\right\Vert ^{2}=\sqint\delta_{a}(x)^{2}\,dx=\delta_{a}(a)=\eta_{a}^{-1}.$}
\end{enumerate}
\end{prop}
\begin{proof}
(i) is an immediate consequence of the definition of $\delta$-basis. 

(ii) By Theorem \ref{thm:sette}, it follows that:
\[
\delta_{a}(x)=\sum_{q\in\Gamma}\frac{1}{\eta_{q}}\left(\sqint\delta_{a}(\xi)\sigma_{q}(\xi)d\xi\right)\sigma_{q}(x)=\sum_{q\in\Gamma}\frac{1}{\eta_{q}}\delta_{aq}\sigma_{q}(x)=\frac{1}{\eta_{a}}\sigma_{a}(x).
\]

(iii) Is an immediate consequence of (ii).
\end{proof}

\subsection{The canonical extension of functions}

We have seen that every function $f:\ \Omega\rightarrow\mathbb{R}$
has a natural extension $f^{*}:\ \Omega^{*}\rightarrow\mathbb{R}^{*}$.
However, in general, $f^{*}$ is not an ultrafunction; in fact, it
is not difficult to prove that the natural extension $f^{\ast}$ of
a function $f$, is an ultrafunction if and only if $f\in V(\Omega).$
So it is useful to define an ultrafunction $f\text{\textdegree}\in V_{\Lambda}(\Omega)$
which approximates $f^{*}$. More in general, for any internal function
$u:\ \Omega^{*}\rightarrow\mathbb{R}^{*}$, we will define an ultrafunction
$u\text{\textdegree}$ as follows:
\begin{defn}
If $u:\ \Omega^{*}\rightarrow\mathbb{R}^{*}$ is an internal function,
we define $u\text{\textdegree}\in V_{\Lambda}(\Omega)$ by the following
formula:
\[
u\text{\textdegree}(x)=\sum_{q\in\Gamma}u(q)\sigma_{q}(x);
\]
if $f:\ \Omega\rightarrow\mathbb{R}$, with some abuse of notation,
we set
\[
f\text{\textdegree}(x)=\left(f^{*}\right)\text{\textdegree}(x)=\sum_{q\in\Gamma}f^{*}(q)\sigma_{q}(x).
\]
\end{defn}
\medskip{}

Since $\Omega\subset\Gamma$, for any internal function $u$, we have
that 
\[
\forall x\in\Omega,\,\,u(x)=u\text{\textdegree}(x)
\]
and
\[
\forall x\in\Omega^{*},\,\,u(x)=u\text{\textdegree}(x)\,\,\Longleftrightarrow u\in V_{\Lambda}(\Omega).
\]

Notice that 
\begin{equation}
P\text{\textdegree}:\ \mathfrak{F}(\Omega)^{*}\rightarrow V_{\Lambda}(\Omega)\label{eq:pia}
\end{equation}
defined by $P\text{\textdegree}(u)=u\text{\textdegree}$ is noting
else but the \emph{orthogonal} projection of $u\in\mathfrak{F}(\Omega)^{*}$
with respect to the semidefinite bilinear form 
\[
\sqint u(x)h(x)dx.
\]

\begin{example}
\label{exa:exa}If \emph{$f\in\mathcal{\mathscr{C}}(\mathbb{R}^{N})$,
and $E\in\mathfrak{C(\textrm{\ensuremath{\Omega}})}$, then $f\theta_{E}\in V(\Omega)$
and hence
\[
\left(f\theta_{E}\right)^{\text{\textdegree}}=f^{*}\theta_{E}^{*}
\]
}
\end{example}
\begin{defn}
\label{def:lina}If a function $f$ is not defined on a set $S:=\Omega\setminus\Theta$,
by convention, we define
\[
f\text{\textdegree}(x)=\sum_{q\in\Gamma\cap\Theta^{*}}f^{*}(q)\sigma_{q}(x).
\]

\medskip{}
\end{defn}
\begin{example}
By the definition above, $\forall x\in\Gamma$, we have that 
\[
\left(\frac{1}{|x|}\right)^{\text{\textdegree}}=\begin{cases}
\frac{1}{|x|} & if\ x\neq0\\
0 & if\ x=0.
\end{cases}
\]
\end{example}
If $f\in\mathscr{C}(\Omega)$, then $f\text{\textdegree}\neq f^{*}$
unless $f\in V_{\Lambda}(\Omega)$. Let examine what $f\text{\textdegree}$
looks like. 
\begin{thm}
Let $f:\,\Omega\rightarrow\mathbb{R}$ be \textup{\emph{continuous}}
in \textup{\emph{a bounded open set}}\textup{ $A\subset\Omega$.}\textup{\emph{
Then, $\forall x\in A^{*},$ with $\mathfrak{mon}(x)\subset A^{*}$
we have that
\[
f\text{\textdegree}(x)=f^{*}(x).
\]
}}
\end{thm}
\begin{proof}
Fix $x_{0}\in A$. Since $A$ is bounded, there exists a set $E\in\mathfrak{C}(\mathrm{\Omega})$
such that 
\[
\mathfrak{mon}(x_{0})\subset E^{*}\subset A^{*}.
\]
We have that (see Example \ref{exa:exa})
\begin{align*}
f\text{\textdegree}(x) & =\sum_{a\in\Gamma}f^{*}(a)\sigma_{a}(x)\\
 & =\sum_{a\in\Gamma}f^{*}(a)\theta_{E}^{*}(a)\sigma_{a}(x)+\sum_{a\in\Gamma}f^{*}(a)(1-\theta_{E}^{*}(a))\sigma_{a}(x)\\
 & =f^{*}(x)\theta_{E}^{*}(x)+\sum_{a\in\Gamma\setminus E^{*}}f^{*}(a)(1-\theta_{E}^{*}(a))\sigma_{a}(x).
\end{align*}
Since $x_{0}\in E^{*}$, $\theta_{E}^{*}(x_{0})=1$; moreover, since
$\mathfrak{mon}(x_{0})\subset E^{*}$, by definition \ref{def:scorre},
(ii), 
\[
\forall a\in\Gamma\setminus E^{*},\,\sigma_{a}(x_{0})=\sigma_{x_{0}}(a)=0.
\]
 Then
\[
f\text{\textdegree}(x_{0})=f^{*}(x_{0}).
\]
\end{proof}
\begin{cor}
If $f\in\mathscr{C}(\Omega)$, then, for any $x\in\Omega^{*}$, such
that $|x|$ is finite, we get 
\[
f\text{\textdegree}(x)=f^{*}(x).
\]
\end{cor}

\subsection{Canonical splitting of an ultrafunction}

In many applications, it is useful to split an ultrafunction $u$
in a part $w\text{\textdegree}$ which is the canonical extension
of a standard function $w$ and a part $\psi$ which is not directly
related to any classical object.

If $u\in V_{\Lambda}(\Omega)$, we set
\[
\Xi=\left\{ x\in\Omega\,|\ u(x)\,\,is\,infinite\right\} 
\]
and
\[
\overline{w}(x)=\begin{cases}
st(u(x)) & if\ x\in\Omega\setminus\Xi\\
0 & if\ x\in\Xi.
\end{cases}
\]

\begin{defn}
For every ultrafunction $u$ consider the splitting
\[
u=w\text{\textdegree}+\psi
\]
where

\begin{itemize}
\item $w=\overline{w}_{|\Omega\setminus\Xi}$ and $w\text{\textdegree}$
which is defined by Definition \ref{def:lina}, is called the \textbf{functional
part} of $u$;
\item $\psi:=u-w\text{\textdegree}$ is called the \textbf{singular part}
of $u$.
\end{itemize}
We will refer to 
\[
S:=\left\{ x\in\Omega^{*}\,|\ \psi(x)\nsim0\right\} 
\]
 as to the \textbf{singular set} of the ultrafunction $u$. 
\end{defn}
\medskip{}
Notice that $w\text{\textdegree}$, the functional part of $u$, may
assume infinite values, but they are determined by the values of $w$,
which is a standard function defined on $\Omega\setminus\Xi$.
\begin{example}
Take $\varepsilon\sim0,$ and 
\[
u(x)=\frac{1}{x^{2}+\varepsilon^{2}}.
\]
In this case
\begin{itemize}
\item $w(x)=\frac{1}{x^{2}},$
\item $\psi(x)=\begin{cases}
-\frac{\varepsilon^{2}}{x^{2}(x^{2}+\varepsilon^{2})} & if\ x\neq0\\
\frac{1}{\varepsilon^{2}} & if\ x=0,
\end{cases}$
\item $S:=\left\{ x\in\mathbb{R}^{*}\,|\ \psi(x)\nsim0\right\} \subset\mathfrak{mon}(0)$.
\end{itemize}
\end{example}
We conclude this section with the following trivial propositions which,
nevertheless, are very useful in applications: 
\begin{prop}
\label{prop:carola}Let $W$ be a Banach space such that $\mathscr{D}(\Omega)\subset W\subseteq L_{loc}^{1}(\Omega)$
and assume that \textup{$u_{\lambda}\in V_{\lambda}$}\textup{\emph{
is weakly convergent in $W$; then if}}
\[
u=w\text{\textdegree}+\psi
\]
 is the canonical splitting\textup{\emph{ of }}$u:=\lim_{\lambda\uparrow\Lambda}\ u_{\lambda}$,
\textup{\emph{there exists a subnet $u_{n}:=u_{\lambda_{n}}$ such
that}}\textup{ }
\[
\lim_{n\rightarrow\infty}u_{n}=w\,\,\,\,weakly\,\,in\,\,W
\]
and 
\[
\forall v\in W,\,\,\sqint\psi v\,dx\sim0.
\]
\textup{\emph{Moreover, if}}\textup{ }
\[
\lim_{n\rightarrow\infty}\left\Vert u_{n}-w\right\Vert _{W}=0
\]
then \textup{$\left\Vert \psi\right\Vert _{W}\sim0.$ }
\end{prop}
\begin{proof}
It is an immediate consequence of Proposition \ref{prop:Ass}. 
\end{proof}
If we use the notation introduced in Definition \ref{def:St}, the
above proposition can be reformulated as follows:
\begin{prop}
If \textup{$u_{\lambda}\in V_{\lambda}$}\textup{\emph{ is weakly
convergent to $w$ in $W$ and }}$u:=\lim_{\lambda\uparrow\Lambda}\ u_{\lambda}$,
\textup{\emph{then }}
\[
w=St_{W_{weak}}(u).
\]
If \textup{$u_{\lambda}$}\textup{\emph{ is strongly convergent to
$w$ in $W$ then}}
\[
w=St_{W}(u).
\]
\end{prop}
An immediate consequence of Proposition \ref{prop:carola} is the
following:
\begin{cor}
If $w\in L^{1}(\Omega)$ then
\[
\sqint w\text{\textdegree}(x)dx\sim\intop w(x)dx.
\]
\end{cor}
\begin{proof}
Since $V_{\Lambda}(\Omega)$ is dense in $L^{1}(\Omega)$ there is
a sequence $u_{n}\in V_{\Lambda}(\Omega)$ which converges strongly
to $w$ in $L^{1}(\Omega)$. Now set 
\[
u:=\lim_{\lambda\uparrow\Lambda}\ u_{|\lambda|}.
\]
By Proposition \ref{prop:carola}, we have that 
\[
u=w\text{\textdegree}+\psi
\]
with $\left\Vert \psi\right\Vert _{L^{1*}}\sim0.$ Since $u$ and
$w\text{\textdegree}$ are in $V_{\Lambda}(\Omega)$, then also $\psi\in V_{\Lambda}(\Omega)$,
so that $\sqint\psi dx=\int^{*}\psi dx\sim0$. Then
\[
\sqint u\text{\emph{(x)dx}\ensuremath{\sim}}\sqint w\text{\textdegree\emph{(x)dx}}.
\]
On the other hand, 
\begin{align*}
\sqint u\text{\emph{(x)dx}} & =\int^{*}u\text{\emph{(x)dx}}=\lim_{\lambda\uparrow\Lambda}\ \int u_{|\lambda|}dx\\
 & \sim\lim_{\lambda\rightarrow\Lambda}\ \int u_{|\lambda|}dx=\lim_{n\rightarrow\infty}\ \int_{\Omega}u_{n}dx=\intop w(x)dx.
\end{align*}
\end{proof}

\section{Differential calculus for ultrafunctions\label{sec:Differential-calculus}}

In this section, we will equip the Caccioppoli space of ultrafunctions
$V_{\Lambda}(\Omega)$ with a suitable notion of derivative which
generalizes the distributional derivative. Moreover we will extend
the Gauss' divergence theorem to the environment of ultrafunctions
and finally we will show the relationship between ultrafunctions and
distributions.

\subsection{The generalized derivative\label{subsec:gelsomina}}

If $u\in V_{\Lambda}(\Omega)\cap\left[C^{1}(\Omega)\right]^{*},$
then, $\partial_{i}^{*}u$ is well defined and hence, using Definition
\ref{def:lina}, we can define an operator
\[
D_{i}\,:\,V_{\Lambda}(\Omega)\cap\left[C^{1}(\Omega)\right]^{*}\rightarrow\,\,V_{\Lambda}(\Omega)
\]
as follows
\begin{equation}
D_{i}u\text{\textdegree}=\left(\partial_{i}^{*}u\right)\text{\textdegree}.\label{eq:ellera-1}
\end{equation}
However it would be useful to extend the operator $D_{i}$ to all
the ultrafunctions in $V_{\Lambda}(\Omega)$ to include in the theory
of ultrafunctions also the weak derivative. Moreover such an extension
allows to compare ultrafunctions with  distributions. In this section
we will define the properties that a generalized derivative must have
(Definition \ref{def:deri}) and in section \ref{sec:FSU}, we will
show that these properties are consistent; we will do that by a construction
of the generalized derivative.
\begin{defn}
\label{def:deri}The generalized derivative 
\[
D_{i}:\ V_{\Lambda}(\Omega)\rightarrow V_{\Lambda}(\Omega)
\]
is an operator defined on a Caccioppoli ultrafunction space $V_{\Lambda}(\Omega)$
which satisfies the following properties:

\begin{enumerate}[label=\Roman*.]
\item $V_{\Lambda}$ has $\sigma$-basis $\left\{ \sigma_{a}(x)\right\} _{a\in\Gamma}$,
such that $\forall a\in\Gamma$ the support of $D_{i}\sigma_{a}$
is contained in $\mathfrak{mon}(a)$;
\item if $u\in V_{\Lambda}(\Omega)\cap\left[C^{1}(\Omega)\right]^{*}$,
then, 
\begin{equation}
D_{i}u\text{\textdegree}=\left(\partial_{i}^{*}u\right)\text{\textdegree};\label{eq:ellera}
\end{equation}
\item $\forall u,v\in V_{\Lambda}(\Omega)$, 
\[
\sqint D_{i}uv\,dx=-\sqint uD_{i}v\,dx
\]
\item if $E\in\mathfrak{C}(\Omega)$, then $\forall v\in V_{\Lambda}(\Omega)$,
\[
\sqint D_{i}\theta_{E}v\,dx=-\int_{\partial E}^{*}v\,(\mathbf{e}_{i}\cdot\mathbf{n}_{E})\,dS
\]
where $\mathbf{n}_{E}$ is the measure theoretic unit outer normal,
integrated on the reduced boundary of $E$ with respect to the $(n-1)$-Hausdorff
measure $dS$ (see e.g. \cite[Section 5.7]{eva-gar}) and $(\mathbf{e}_{1},....,\mathbf{e}_{N})$
is the canonical basis of $\mathbb{R}^{N}$.
\end{enumerate}
\end{defn}
We remark that, in the framework of the theory of Caccioppoli sets,
the classical formula corresponding to IV is the following: $\forall v\in\mathfrak{\mathscr{C}}(\Omega)$,
\[
\int\partial_{i}\theta_{E}v\,dx=-\int_{\partial E}v\,(\mathbf{e}_{i}\cdot\mathbf{n}_{E})\,dS.
\]

\medskip{}
The existence of a generalized derivative will be proved in section
\ref{sec:FSU}.

Now let us define some differential operators: 
\begin{itemize}
\item $\nabla=(\partial_{1},...,\partial_{N})$ will denote the usual gradient
of standard functions;
\item $\nabla^{*}=(\partial_{1}^{*},...,\partial_{N}^{*})$ will denote
the natural extension of the gradient (in the sense of NSA);
\item $D=(D_{1},...,D_{N})$ will denote the canonical extension of the
gradient in the sense of the ultrafunctions (Definition \ref{def:deri}).
\end{itemize}
Next let us consider the divergence:
\begin{itemize}
\item $\nabla\cdot\phi=\partial_{1}\phi_{1}+...+\partial_{N}\phi_{N}$ will
denote the usual divergence of standard vector fields $\phi\in\left[\mathscr{C}^{1}(\overline{\Omega})\right]^{N}$;
\item $\nabla^{*}\cdot\phi=\partial_{1}^{*}\phi_{1}+...+\partial_{N}^{*}\phi_{N}$
will denote the divergence of internal vector fields $\phi\in\left[\mathscr{C}^{1}(\overline{\Omega})^{*}\right]^{N}$;
\item $D\cdot\phi$ will denote the unique ultrafunction $D\cdot\phi\in V_{\Lambda}(\Omega)$
such that, $\forall v\in V_{\Lambda}(\Omega),$
\begin{equation}
\sqint D\cdot\phi vdx=-\sqint\phi(x)\cdot Dv\,dx.\label{eq:rosa}
\end{equation}
\end{itemize}
And finally, we can define the Laplace operator:
\begin{itemize}
\item $\bigtriangleup\text{\textdegree}$ or $D^{2}$ will denote the Laplace
operator defined by $D\circ D$.
\end{itemize}

\subsection{The Gauss' divergence theorem }

By Definition \ref{def:deri}, IV, for any set $E\in\mathfrak{C}_{\Lambda}(\Omega)$
and $v\in V_{\Lambda}(\Omega)$,
\[
\sqint D_{i}\theta_{E}v\,dx=-\int_{\partial E}^{*}v\,(\mathbf{e}_{i}\cdot\mathbf{n}_{E})\,dS,
\]
and by Definition \ref{def:deri}, III, 
\[
\sqint D_{i}v\,\theta_{E}dx=\int_{\partial E}^{*}v\,(\mathbf{e}_{i}\cdot\mathbf{n}_{E})\,dS.
\]
Now, if we take a vector field $\phi=(v_{1},...,v_{N})\in\left[V_{\Lambda}(\Omega)\right]^{N}$,
by the above identity, we get
\begin{equation}
\sqint D\cdot\phi\,\theta_{E}\,dx=\int_{\partial E}^{*}\phi\cdot\mathbf{n}_{E}\,dS.\label{eq:camilla}
\end{equation}
Now, if $\phi\in\mathscr{C}^{1}(\overline{\Omega})$ and $\partial E$
is smooth, we get the divergence Gauss' theorem:
\[
\int_{E}\nabla\cdot\phi\,dx=\int_{\partial E}\phi\cdot\mathbf{n}_{E}\,dS.
\]

Then, (\ref{eq:camilla}) is a generalization of the Gauss' theorem
which makes sense for any set $E\in\mathfrak{C}_{\Lambda}(\Omega).$
Next, we want to generalize Gauss' theorem to any subset $A\subset\Omega$. 

First of all we need to generalize the notion of Caccioppoli perimeter
$p(E)$ to any arbitrary set. As we have seen in Section \ref{subsec:Cacc},
if $E\in\mathfrak{C}(\mathbb{\mathrm{\Omega}})$ is a special Caccioppoli
set, we have that 
\[
p(E)=\int|\nabla\theta_{E}|\,dx,
\]
and it is possible to define a $(n-1)$-dimensional measure $dS$
as follows
\[
\int_{\partial E}v(x)\,dS:=\int|\nabla\theta_{E}|\,v(x)\,dx.
\]
In particular, if the reduced boundary of $E$ coincides with $\partial E,$
we have that (see \cite[Section 5.7]{eva-gar})
\[
\int_{\partial E}v(x)\,dS=\int_{\partial E}v(x)\,d\mathcal{H}^{N-1}.
\]

Then, the following definition is a natural generalization:
\begin{defn}
\label{def:misura}If $A$ is a measurable subset of $\Omega$, we
set
\[
p(A):=\sqint|D\theta_{A}^{\text{\textdegree}}|\,dx
\]
and $\forall v\in V_{\Lambda}(\Omega),$
\begin{equation}
\sqint_{\partial A}v(x)\,dS:=\sqint v(x)\,|D\theta_{A}^{\text{\textdegree}}|\,dx.\label{eq:nina}
\end{equation}
\end{defn}
\begin{rem}
Notice that
\[
\sqint_{\partial A}v(x)\,dS\neq\sqint v(x)\chi_{\partial A}^{\text{\textdegree}}(x)\,dx.
\]
In fact the left hand term has been defined as follows:
\[
\sqint_{\partial A}v(x)\,dS=\sum_{x\in\Gamma}v(x)\,|D\theta_{A}^{\text{\textdegree}}(x)|\,\eta_{x}
\]
while the right hand term is
\[
\sqint v(x)\chi_{\partial A}^{\text{\textdegree}}(x)\,dx=\sum_{x\in\Gamma}v(x)\chi_{\partial A}^{\text{\textdegree}}(x)\,\eta_{x},
\]
in particular if $\partial A$ is smooth and $v(x)$ is bounded, $\sum_{_{x\in\Gamma}}v(x)\chi_{\partial A}^{\text{\textdegree}}(x)\,\eta_{x}$
is an infinitesimal number.
\end{rem}
\begin{thm}
\label{thm:41}If $A$ is an arbitrary measurable subset of $\Omega$,
we have that
\begin{equation}
\sqint D\cdot\phi\,\theta_{A}^{\text{\textdegree}}\,dx=\sqint_{\partial A}\phi\cdot\mathbf{n}_{A}^{\text{\textdegree}}(x)\,dS,\label{eq:gauus}
\end{equation}
where
\[
\mathbf{n}_{A}^{\text{\textdegree}}(x)=\begin{cases}
-\frac{D\theta_{A}^{\text{\textdegree}}(x)}{|D\theta_{A}^{\text{\textdegree}}(x)|} & if\,\,\,D\theta_{A}^{\text{\textdegree}}(x)\neq0\\
0 & if\,\,\,D\theta_{A}^{\text{\textdegree}}(x)=0
\end{cases}
\]
\end{thm}
\begin{proof}
By Theorem \ref{def:deri}, III, 

\[
\sqint D\cdot\phi\,\theta_{A}^{\text{\textdegree}}dx=-\sqint\phi\cdot D\theta_{A}^{\text{\textdegree}}dx,
\]
then, using the definition of $\mathbf{n}_{A}^{\text{\textdegree}}(x)$
and (\ref{eq:nina}), the above formula can be written as follows:
\[
\sqint D\cdot\phi\,\theta_{A}^{\text{\textdegree}}dx=\sqint\phi\cdot\mathbf{n}_{A}^{\text{\textdegree}}\:|D\theta_{A}^{\text{\textdegree}}|\,dx=\sqint_{\partial A}\phi\cdot\mathbf{n}_{A}^{\text{\textdegree}}\,dS.
\]
\end{proof}
Clearly, if $E\in\mathfrak{C}_{\Lambda}(\Omega)$, then
\[
\sqint_{\partial E}\phi\cdot\mathbf{n}_{E}^{\text{\textdegree}}\,dS=\int_{\partial E}\phi\cdot\mathbf{n}_{E}\,dS.
\]

\begin{example}
If $A$ is the Koch snowflake, then the usual Gauss' theorem makes
no sense since $p(A)=+\infty$; on the other hand equation (\ref{eq:gauus})
holds true. Moreover, the perimeter in the sense of ultrafunction
is an infinite number given by Definition \ref{def:misura}. In general,
if $\partial A$ is a $d$-dimensional fractal set, it is an interesting
open problem to investigate the relation between its Hausdorff measure
and the ultrafunction ``measure'' $dS=|D\theta_{A}^{\text{\textdegree}}|\,dx$.
\end{example}

\subsection{Ultrafunctions and distributions}

One of the most important properties of the ultrafunctions is that
they can be seen (in some sense that we will make precise in this
section) as generalizations of the distributions.
\begin{defn}
\label{DEfCorrespondenceDistrUltra}The space of \textbf{generalized
distribution} on $\Omega$ is defined as follows: 
\[
\mathcal{\mathscr{D}}_{G}^{\prime}(\Omega)=V_{\Lambda}(\Omega)/N,
\]
where 
\[
N=\left\{ \tau\in V_{\Lambda}(\Omega)\ |\ \forall\varphi\in\mathscr{D}(\Omega),\ \int\tau\varphi\ dx\sim0\right\} .
\]
\end{defn}
The equivalence class of $u$ in $V_{\Lambda}(\Omega)$ will be denoted
by 
\[
\left[u\right]_{\mathscr{D}}.
\]

\begin{defn}
Let $\left[u\right]_{\mathfrak{\mathscr{D}}}$ be a generalized distribution.
We say that $\left[u\right]_{\mathscr{D}}$ is a bounded generalized
distribution if, $\forall\varphi\in\mathfrak{\mathcal{\mathscr{D}}}(\Omega),\,\int u\varphi^{*}\ dx\ \ \text{is\ finite}$.
We will denote by $\mathscr{D}_{GB}^{\prime}(\Omega)$ the set of
the bounded generalized distributions.
\end{defn}
We have the following result. 
\begin{thm}
\label{bello}There is a linear isomorphism 
\[
\Phi:\mathfrak{\mathcal{\mathscr{D}}}_{GB}^{\prime}(\Omega)\rightarrow\mathfrak{\mathscr{D}}^{\prime}(\Omega),
\]
defined by 
\[
\left\langle \Phi\left(\left[u\right]_{\mathfrak{\mathcal{\mathscr{D}}}}\right),\varphi\right\rangle _{\mathfrak{\mathcal{\mathscr{D}}}(\Omega)}=st\left(\sqint u\,\varphi^{\ast}\ dx\right).
\]
\end{thm}
\begin{proof}
For the proof see e.g. \cite{algebra}.
\end{proof}
From now on we will identify the spaces $\mathfrak{\mathscr{D}}_{GB}^{\prime}(\Omega)$
and $\mathscr{D}^{\prime}(\Omega);$ so, we will identify $\left[u\right]_{\mathscr{D}}$
with $\Phi\left(\left[u\right]_{\mathscr{D}}\right)$ and we will
write $\left[u\right]_{\mathscr{D}}\in\mathfrak{\mathscr{D}}^{\prime}(\Omega)$
and 
\[
\left\langle \left[u\right]_{\mathscr{D}},\varphi\right\rangle _{\mathscr{D}(\Omega)}:=\langle\Phi[u]_{\mathscr{D}},\varphi\rangle_{\mathscr{D}(\Omega)}=st\left(\sqint u\ \varphi^{\ast}\ dx\right).
\]

Moreover, with some abuse of notation, we will write also that $\left[u\right]_{\mathscr{D}}\in L^{2}(\Omega),\ \left[u\right]_{\mathfrak{\mathscr{D}}}\in V(\Omega),$
etc. meaning that the distribution $\left[u\right]_{\mathscr{D}}$
can be identified with a function $f$ in $L^{2}(\Omega),$ $V(\Omega),$
etc. By our construction, this is equivalent to say that $f^{\ast}\in\left[u\right]_{\mathfrak{\mathscr{D}}}.$
So, in this case, we have that $\forall\varphi\in\mathfrak{\mathscr{D}}(\Omega)$,
\[
\left\langle \left[u\right]_{\mathfrak{\mathscr{D}}},\varphi\right\rangle _{\mathfrak{\mathscr{D}}(\Omega)}=st\left(\int^{\ast}u\ \varphi^{\ast}\ dx\right)=st\left(\int^{\ast}f^{\ast}\varphi^{\ast}dx\right)=\int f\ \varphi\ dx.
\]

\begin{rem}
\label{rem:Schw}Since an ultrafunction $u:\,\Omega^{*}\rightarrow\mathbb{R}^{*}$
is univocally determined by its value in $\Gamma$, we may think of
the ultrafunction as being defined only on $\Gamma$ and to denote
them by $V_{\Lambda}(\Gamma)$; the set $V_{\Lambda}(\Gamma)$ is
an algebra which extends the algebra of continuous functions $\mathscr{C}(\Omega)$
if it is equipped with the pointwise product. 

Moreover, we recall that, by a well known theorem of Schwartz, any
tempered distribution can be represented as $\partial^{\alpha}f$,
where $\alpha$ is a multi-index and $f$ is a continuous function.
If we identify $T=\partial^{\alpha}f$ with the ultrafunction $D^{\alpha}f\text{\textdegree}$,
we have that the set of tempered distributions $\mathscr{S}'$ is
contained in $V_{\Lambda}(\Gamma)$. However the Schwartz impossibility
theorem (see introduction) is not violated since $(V_{\Lambda}(\Gamma),+,\,\cdot\,,\,D)$
is not a differential algebra, since the Leibnitz rule does not hold
for some couple of ultrafunctions.
\end{rem}

\section{Construction of the Caccioppoli space of ultrafunctions\label{sec:FSU}}

In this section we will prove the existence of Caccioppoli spaces
of ultrafunctions (see Definition \ref{def:scorre}) by an explicit
construction.

\subsection{Construction of the space $V_{\Lambda}(\Omega)$}

In this section we will construct a space of ultrafunctions $V_{\Lambda}(\Omega)$
and in the next section we will equip it with a $\sigma$-basis in
such a way that $V_{\Lambda}(\Omega)$ becomes a Caccioppoli space
of ultrafunctions according to Definition \ref{def:scorre}.
\begin{defn}
Given a family of open sets $\mathfrak{R}_{0}$ , we say that a family
of open sets $\mathfrak{B}=\left\{ E_{k}\right\} _{k\in K}$ is a
\emph{basis} for $\mathfrak{R}_{0}$ if 
\begin{itemize}
\item $\forall k\neq h,\,E_{k}\cap E_{h}=\emptyset;$ 
\item $\forall A\in\mathfrak{R}_{0}$, there is a set of indices \emph{$K_{E}\subset K$}
such that\emph{
\begin{equation}
A=int\left(\bigcup_{k\in K_{E}}\overline{E_{k}}\right);\label{eq:giulia}
\end{equation}
}
\item $\mathfrak{B}$ is the smallest family of sets which satisfies the
above properties.
\end{itemize}
We we will refer to the family $\mathfrak{R}$ of \textbf{all} the
open sets which can be written by the espression (\ref{eq:giulia})\emph{
}as to the family\emph{ generated} by $\mathfrak{R}_{0}$. 
\end{defn}
Let us verify that 
\begin{lem}
For any finite family of special Caccioppoli sets $\mathfrak{C}_{0}$,
there exists a basis $\mathfrak{B}$ whose elements are special Caccioppoli
sets. Moreover also the set $\mathfrak{C}$ generated by $\mathfrak{C}_{0}$
consists of special Caccioppoli sets.
\end{lem}
\begin{proof}
For any $x\in\Omega$, we set 
\[
E_{x}=\bigcap\left\{ A\in\mathfrak{C}_{0}\:|\,x\in A\right\} .
\]
We claim that $\left\{ E_{x}\right\} _{x\in\Omega}$ is a basis. Since
$\mathfrak{C}_{0}$ is a finite family, then also $\left\{ E_{x}\right\} _{x\in\Omega}$
is a finite family and hence there is a finite set of indices $K$
such that $\mathfrak{B}=\left\{ E_{k}\right\} _{k\in K}$. Now it
is easy to prove that $\mathfrak{B}$ is a basis and it consists of
special Caccioppoli sets. Also the last statement is trivial.
\end{proof}
We set 
\[
\mathfrak{C}_{0,\lambda}(\Omega):=\lambda\cap\mathfrak{C}(\Omega),
\]
and we denote by $\mathfrak{B}_{\lambda}(\Omega)$ and $\mathfrak{C}_{\lambda}(\Omega)$
the relative basis and the generated family which exist by the previous
lemma.

\medskip{}
Now set
\begin{equation}
\mathfrak{C}_{\Lambda}(\Omega)=\lim_{\lambda\uparrow\Lambda}\mathfrak{C}_{\lambda}(\Omega),\ \mathfrak{B}_{\Lambda}(\Omega)=\lim_{\lambda\uparrow\Lambda}\mathfrak{B}_{\lambda}(\Omega).\label{eq:laura}
\end{equation}

\begin{lem}
\label{lem:Wee}The following properties hold true

\begin{itemize}
\item $\mathfrak{C}_{\Lambda}(\Omega)$ and $\mathfrak{B}_{\Lambda}(\Omega)$
are hyperfinite;
\item $\mathfrak{C}(\Omega)^{\sigma}\subset\mathfrak{C}_{\Lambda}(\Omega)\subset\mathfrak{C}(\Omega)^{*}$;
\item if $E\in\mathfrak{C}_{\Lambda}(\Omega)$, then
\[
\theta_{E}=\sum_{Q\in K(E)}\theta_{Q}(x),
\]
\textup{where $K(E)\subset\mathfrak{B}_{\Lambda}(\Omega)$ is a hyperfinite
set and $\theta_{Q}$ is the natural extension to} $\mathfrak{C}_{\Lambda}(\Omega)^{*}$\textup{
of the function $Q\mapsto\theta_{Q}$ defined on} $\mathfrak{C}_{\Lambda}(\Omega)$\textup{
by (\ref{eq:estate}).}
\end{itemize}
\end{lem}
\begin{proof}
It follows trivially by the construction.
\end{proof}
The next lemma is a basic step for the construction of the space $V_{\Lambda}(\Omega)$.
\begin{lem}
\label{thm:dodici}For any\emph{ }$Q\in\mathfrak{B}_{\Lambda}(\Omega)$
there exists a set \textup{\emph{$\Xi(Q)\subset\overline{Q}\cap\Omega,$}}
and a family of functions\textup{ $\left\{ \zeta_{a}\right\} _{a\in\Xi(Q)}$,}\textup{\emph{
such that,}}
\begin{enumerate}
\item \textup{$\Xi:=\bigcup\left\{ \Xi(Q)\ |\ Q\in\mathfrak{B}_{\Lambda}(\Omega)\right\} $
}\textup{\emph{is a hyperfinite set, and }}$\Omega\subset\Xi\subset\Omega^{*}$;
\item if $Q,R\in\mathfrak{B}_{\Lambda}(\Omega)$ and $Q\neq R$, then \textup{$\Xi(Q)\cap\Xi(R)=\emptyset$;}
\item if $a\in\Xi(Q)$, the\textup{\emph{n, $\exists f_{a}\in\mathscr{C}^{1}(\Omega)^{*}$}}\textup{
such that $\zeta_{a}=f_{a}\cdot\theta_{Q};$}
\item \textup{\emph{for any $a,b\in\Xi$, $a\neq b$ implies $\mathfrak{supp}(\zeta_{a})\cap\mathfrak{supp}(\zeta_{b})=\emptyset;$}}
\item \emph{$\zeta_{a}\geq0$;}
\item \textup{\emph{for any $a\in\Xi$,}}
\begin{equation}
\zeta_{a}(a)=1.\label{eq:alla}
\end{equation}
\end{enumerate}
\end{lem}
\begin{proof}
We set
\begin{equation}
r(\lambda)=\frac{1}{3}min\left\{ d\left(x,y\right)\ |\ x,y\in\lambda\cap\Omega\right\} ,\label{eq:pinta-2}
\end{equation}
and we denote by $\rho$ a smooth bell shaped function having support
in $B_{1}(0)$; then the functions $\rho\left(\frac{x-a_{\lambda}}{r(\lambda)}\right),\,a_{\lambda}\in\lambda\cap\Omega$
have disjoint support. We set
\[
\Xi:=\left\{ \lim_{\lambda\uparrow\Lambda}a_{\lambda}\,|\ a_{\lambda}\in\lambda\cap\Omega\right\} ,
\]
so that $\Omega\subset\Xi\subset\Omega^{*}$ and we divide all points
$a\in\Xi,$ among sets $\Xi(Q)$, $Q\in\mathfrak{B}_{\Lambda}$, in
such a way that

- if $a\in Q$ then $a\in\Xi(Q)$;

- if $a\in\partial Q_{1}\cap...\cap\partial Q_{l}$ there exists a
unique $Q_{j}$ ($j\leq l$) such that $a\in\Xi(Q_{j})$.

With this construction, claims 1. and 2. are trivially satisfied.
Now, for any $a\in\Xi(Q)$, set 
\[
\rho_{a}(x):=\lim_{\lambda\uparrow\Lambda}\rho\left(\frac{x-a_{\lambda}}{r(\lambda)}\right),
\]
and
\begin{equation}
\zeta_{a}(x):=\frac{\rho_{a}(x)\theta_{Q}(x)}{\rho_{a}(a)\theta_{Q}(a)}.\label{eq:grillo}
\end{equation}
It is easy to check that the functions $\zeta_{a}$ satisfy 3.,4.,5.
and 6. 
\end{proof}
We set
\begin{equation}
V_{\Lambda}^{1}(\Omega)=span\left(\left\{ \zeta_{a}\right\} _{a\in\Xi}\right)+\lim_{\lambda\uparrow\Lambda}\left(\lambda\cap\mathscr{C}^{1}(\overline{\Omega})\right),\,\label{eq:pinta}
\end{equation}
and
\[
V_{\Lambda}^{1}(Q)=\left\{ u\theta_{Q}\ |\ u\in V_{\Lambda}^{1}(\Omega)\right\} ;
\]
so we have that, for any $a\in\Xi(Q)$, $\zeta_{a}\in V_{\Lambda}^{1}(Q)$.
Also, we set

\begin{equation}
V_{\Lambda}^{0}(\Omega)=Span\left(\left\{ f,\,\partial_{i}f,\,fg,\ g\partial_{i}f\,\,|\,\,f,g\in V_{\Lambda}^{1}(\Omega),\,\,i=1,...,N\right\} +\lim_{\lambda\uparrow\Lambda}\left(\lambda\cap\mathscr{C}(\overline{\Omega})\right)\right)\label{eq:pinta-0}
\end{equation}
and
\[
V_{\Lambda}^{0}(Q)=\left\{ u\theta_{Q}\ |\ u\in V_{\Lambda}^{0}(\Omega)\right\} .
\]

\textit{\emph{Finally, we can define }}the $V_{\Lambda}(\Omega)$
as follows:
\begin{equation}
V_{\Lambda}(\Omega)=\underset{Q\in\mathfrak{B}_{\Lambda}(\Omega)}{\bigoplus}V_{\Lambda}^{0}(Q).\label{eq:fs}
\end{equation}
Namely, if $u\in V_{\Lambda}(\Omega)$, then
\begin{equation}
u(x)=\sum_{Q\in\mathfrak{B}_{\Lambda}(\Omega)}u_{Q}(x)\theta_{Q}(x),\label{eq:splitty}
\end{equation}
with $u_{Q}\in V_{\Lambda}^{0}(\Omega)$.

\subsection{The $\sigma$-basis }

In this section, we will introduce a $\sigma$-basis in such a way
that $V_{\Lambda}(\Omega)$ becomes a Caccioppoli space of ultrafunctions,
according to Definition \ref{def:scorre}.
\begin{thm}
\label{papa} There exists a $\sigma$-basis\emph{ }\textup{\emph{for
}}\emph{$V_{\Lambda}(\Omega)$, $\left\{ \sigma_{a}(x)\right\} _{a\in\Gamma}$,}
such that

\begin{enumerate}
\item $\Omega\subset\Gamma\subset\Omega^{*};$
\item $\Gamma=\bigcup_{_{Q\in\mathfrak{B}_{\Lambda}(\Omega)}}Q_{\Gamma},$
where $Q\cap\Gamma\subset Q_{\Gamma}\subset\overline{Q}\cap\Gamma$,
and $Q_{\Gamma}\cap R_{\Gamma}=\emptyset\,\,for\,\,Q\neq R$; 
\item $\left\{ \sigma_{a}(x)\right\} _{a\in Q_{\Gamma}}$ is a\emph{ }$\sigma$-basis\emph{
}\textup{\emph{for }}\emph{$V_{\Lambda}^{0}(Q)$.}
\end{enumerate}
\end{thm}
\begin{proof}
First we introduce in\emph{ $V_{\Lambda}(\Omega)$} the following
scalar product:
\begin{equation}
\left\langle u,v\right\rangle =\int^{*}uv\,dx.\label{eq:pippo}
\end{equation}

For any $Q\in\mathfrak{B}_{\Lambda}(\Omega)$ we set
\[
Z(Q)=\left\{ \sum_{a\in\Xi(Q)}\gamma_{a}\zeta_{a}(x)\ |\ \gamma_{a}\in\mathbb{R}^{*}\right\} ,
\]
where $\Xi(Q)$ and the functions$\left\{ \zeta_{a}\right\} _{a\in\Xi}$
are defined in Lemma \ref{thm:dodici}.

If we set 
\[
\mathfrak{d}_{a}(x)=\frac{\zeta_{a}(x)}{\int^{*}|\zeta_{a}(x)|^{2}\,dx}
\]
we have that
\begin{equation}
\left\{ \mathfrak{d}_{a}(x)\right\} _{a\in\Xi(Q)}\label{eq:belinda}
\end{equation}
is a $\delta$-basis for $Z(Q)\subset V_{\Lambda}^{0}(Q)$ (with respect
to the scalar product (\ref{eq:pippo})). In fact, if $u\in Z(Q)$,
then $u(x)=\sum_{b\in\Xi(Q)}u(b)\zeta_{b}(x)$, and hence, by Lemma
\ref{thm:dodici}, it follows that
\begin{align*}
\int^{*}u(x)\mathfrak{d}_{a}(x)\,dx & =\int^{*}\sum_{b\in\Xi(Q)}u(b)\zeta_{b}(x)\mathfrak{d}_{a}(x)\,dx\\
 & =\sum_{b\in\Xi(Q)}u(b)\int^{*}\zeta_{b}(x)\mathfrak{d}_{a}(x)\,dx\\
 & =\sum_{b\in\Xi(Q)}u(b)\int^{*}\zeta_{b}(x)\left(\frac{\zeta_{a}(x)}{\int^{*}|\zeta_{a}(x)|^{2}\,dx}\,\right)dx\\
 & =\sum_{b\in\Xi(Q)}u(b)\delta_{ab}=u(a).
\end{align*}

Next, we want to complete this basis and to get a $\delta$-basis
for $V_{\Lambda}^{0}(Q)$. To this aim, we take an orthonormal basis
$\left\{ e_{k}(x)\right\} $ of $Z(Q)^{\perp}$ where $Z(Q)^{\perp}$
is the orthogonal complement of $Z(Q)$ in $V_{\Lambda}^{0}(Q)$ (with
respect to the scalar product \eqref{eq:pippo}). For every $a\in Q\backslash\Xi$,
set
\[
\mathfrak{d}_{a}(x)=\sum_{k}e_{k}(a)e_{k}(x);
\]
notice that this definition is not in contradiction with (\ref{eq:grillo})
since in the latter $a\in\Xi$.

For every $v\in Z(Q)^{\perp}$, we have that
\[
\int^{*}v(x)\mathfrak{d}_{a}(x)\,dx=v(a);
\]
in fact
\begin{align*}
\int^{*}v(x)\mathfrak{d}_{a}(x)\,dx & =\int^{*}\left(\sum_{k}v_{k}e_{k}(x)\right)\left(\sum_{h}e_{h}(a)e_{h}(x)\right)\,dx=\sum_{k,h}v_{k}e_{h}(a)\int^{*}e_{k}(x)e_{h}(x)\,dx\\
 & =\sum_{k,h}v_{k}e_{h}(a)\delta_{hk}=\sum_{k}v_{k}e_{k}(a)=v(a).
\end{align*}
It is not difficult to realize that $\left\{ \mathfrak{d}_{a}(x)\right\} _{a\in Q\backslash\Xi}$
generates all $Z(Q)^{\perp}$ and hence we can select a set $\Xi^{\star}(Q)\subset Q\backslash\Xi$
such that $\left\{ \mathfrak{d}_{a}(x)\right\} _{a\in\Xi^{\star}(Q)}$
is a basis for $Z(Q)^{\perp}$. Taking 
\[
Q_{\Gamma}=\Xi^{\star}(Q)\cup\Xi(Q),
\]
we have that $\left\{ \mathfrak{d}_{a}(x)\right\} _{a\in Q_{\Gamma}}$
is a basis for \emph{$V_{\Lambda}^{0}(Q)$.}

Now let $\left\{ \sigma_{a}(x)\right\} _{a\in Q_{\Gamma}}$ denote
the dual basis of $\left\{ \mathfrak{d}_{a}(x)\right\} _{a\in Q_{\Gamma}}$
namely a basis such that, $\forall a,b\in Q_{\Gamma},$
\[
\int^{*}\sigma_{a}(x)\mathfrak{d}_{b}(x)dx=\delta_{ab}.
\]

Clearly it is a $\sigma$-basis for $V_{\Lambda}^{0}(Q)$. In fact,
if $u\in V_{\Lambda}^{0}(Q)$, we have that
\[
u(x)=\sum_{a\in Q_{\Gamma}}\left[\int^{*}u(t)\mathfrak{d}_{a}(t)\,dt\right]\sigma_{a}(x)=\sum_{a\in Q_{\Gamma}}u(a)\sigma_{a}(x).
\]
Notice that if $a\in\Xi(Q)$, then $\sigma_{a}(x)=\zeta_{a}(x)$.
The conclusion follows taking 
\[
\Gamma:=\bigcup_{Q\in\mathfrak{B}_{\Lambda}(\Omega)}Q_{\Gamma}.
\]
\end{proof}
By the above theorem, the following corollary follows straightforward.
\begin{cor}
\label{def:scorre-1}\textup{ $V_{\Lambda}(\Omega)$ is a} \emph{Caccioppoli}
\emph{space of ultrafunctions}\textup{ in the sense of Definition
\ref{def:scorre}.}
\end{cor}
If $E\in\mathfrak{C}_{\Lambda}(\Omega)$ (see (\ref{eq:laura})),
we set
\[
E_{\Gamma}=\bigcup_{Q\in\mathfrak{B}_{\Lambda}(\Omega),Q\subset E}Q_{\Gamma}
\]
If, for any internal set $A$, we define
\[
\sqint_{A}u(x)dx=\sum_{a\in\Gamma\cap A}u(a)\eta_{a}
\]
then, we have the following result:
\begin{thm}
\label{thm:rigoletto}If $u\theta_{E}\in V_{\Lambda}(\Omega)$ and
$E\in\mathfrak{C}_{\Lambda}(\Omega)$, then
\[
\sqint_{E_{\Gamma}}u(x)dx=\int_{E}^{*}u(x)dx=\sqint u(x)\theta_{E}(x)dx.
\]
\end{thm}
\begin{proof}
We have that
\begin{align}
\sqint_{E_{\Gamma}}u(x)dx & =\sum_{a\in E_{\Gamma}}u(a)\eta_{a}=\int^{*}\sum_{a\in E_{\Gamma}}u(a)\sigma_{a}(x)dx\nonumber \\
 & =\int^{*}\sum_{Q\subset E}\,\sum_{a\in Q_{\Gamma}}u(a)\sigma_{a}(x)dx\label{eq:rosamunda}
\end{align}
Since $u\theta_{E}\in V_{\Lambda}(\Omega)$, then, by \eqref{eq:splitty},
we can write 
\[
u(x)\theta_{E}(x)=\sum_{Q\subset E}u_{Q}(x)\theta_{Q}(x).
\]

By Th. \ref{papa}, 3., $u_{Q}(x)\theta_{Q}(x)=\sum_{a\in Q_{\Gamma}}u(a)\sigma_{a}(x)\in V_{\Lambda}^{0}(Q)$.
Then by (\ref{eq:rosamunda})

\begin{align}
\sqint_{E_{\Gamma}}u(x)dx & =\int^{*}\sum_{Q\subset E}\,u_{Q}(x)\theta_{Q}(x)dx=\sum_{Q\subset E}\,\int^{*}u_{Q}(x)\theta_{Q}(x)dx.\label{eq:carla}
\end{align}
By this equation and the fact that $\int^{*}u_{Q}(x)\theta_{Q}(x)dx=\int_{Q}^{*}u(x)dx$,
it follows that

\[
\sqint_{E_{\Gamma}}u(x)dx=\sum_{Q\subset E}\,\int_{Q}^{*}u(x)dx=\int_{E}^{*}u(x)dx.
\]
Moreover, since $u_{Q}\theta_{Q}\in V_{\Lambda}^{0}(Q)\subset V_{\Lambda}(\Omega)$,
\[
\int^{*}u_{Q}(x)\theta_{Q}(x)dx=\sqint u_{Q}(x)\theta_{Q}(x)dx,
\]
by (\ref{eq:carla}) we have
\[
\sqint_{E_{\Gamma}}u(x)dx=\sum_{Q\subset E}\sqint u_{Q}(x)\theta_{Q}(x)dx=\sqint\sum_{Q\subset E}u_{Q}(x)\theta_{Q}(x)dx=\sqint u(x)\theta_{E}(x)dx.
\]
\end{proof}

\subsection{Construction of the generalized derivative}

Next we construct a generalized derivative on $V_{\Lambda}(\Omega)$.

We set\emph{
\[
U_{\Lambda}^{1}=\underset{Q\in\mathfrak{B}_{\Lambda}(\Omega)}{\bigoplus}V_{\Lambda}^{1}(Q),
\]
 }and 
\[
U_{\Lambda}^{0}=\left(U_{\Lambda}^{1}\right)^{\perp},
\]
will denote the the orthogonal complement of $U_{\Lambda}^{1}$ in
$V_{\Lambda}(\Omega)$. According to this decomposition, $V_{\Lambda}(\Omega)=U_{\Lambda}^{1}\oplus U_{\Lambda}^{0}$
and we can define the following orthogonal projectors
\[
P_{i}\,:\,V_{\Lambda}(\Omega)\rightarrow U_{\Lambda}^{i},\,\,i=0,1\,\,,
\]
hence, any ultrafunction $u\in V_{\Lambda}(\Omega)$ has the following
orthogonal splitting: $u=u_{1}+u_{0}$ where $u_{i}=P_{i}u.$

Now we are able to define the generalized partial derivative for $u\in V_{\Lambda}^{1}(\Omega)$.
\begin{defn}
\label{def:bellissima}We define the generalized partial derivative
\[
D_{i}:U_{\Lambda}^{1}(\Omega)\rightarrow V_{\Lambda}(\Omega),
\]
 as follows: 
\begin{equation}
\sqint D_{i}uv\,dx=\sum_{Q\in\mathfrak{B}_{\Lambda}(\Omega)}\sqint\partial_{i}^{*}u_{Q}v_{Q}\theta_{Q}\,dx-\frac{1}{2}\sum_{Q\in\mathfrak{B}_{\Lambda}(\Omega)}\sum_{R\in\mathfrak{Y}(Q)}\int_{\partial Q\cap\partial R}^{*}\left(u_{Q}-u_{R}\right)v_{Q}\,(\mathbf{e}_{i}\cdot\mathbf{n}_{Q})\,dS,\label{eq:g1}
\end{equation}
where 
\[
\mathfrak{Y}(Q)=\left\{ R\in\mathfrak{B}_{\Lambda}(\Omega)\cup\left\{ Q_{\infty}\right\} \,|\,Q\neq R,\,\,\partial Q\cap\partial R\neq\emptyset\right\} 
\]
with 
\[
Q_{\infty}=\Omega^{*}\setminus\overline{\bigcup_{Q\in\mathfrak{B}_{\Lambda}(\Omega)}Q},
\]
Moreover, if $u=u_{1}+u_{0}\in U_{\Lambda}^{1}\oplus U_{\Lambda}^{0}=V_{\Lambda}(\Omega),$
we set
\begin{equation}
D_{i}u=D_{i}u_{1}-\left(D_{i}P_{1}\right)^{\dagger}u_{0},\label{eq:bellissima-2}
\end{equation}
where, for any linear operator $L$, $L^{\dagger}$ denotes the adjoint
operator.
\end{defn}
\medskip{}

\begin{rem}
Notice that, if $u,v\in U_{\Lambda}^{1}$, by Th. \ref{thm:dodici},2.,
we have
\begin{equation}
\sqint D_{i}uv\,dx=\sum_{Q\in\mathfrak{B}_{\Lambda}(\Omega)}\int_{Q}^{*}\partial_{i}^{*}u_{Q}v_{Q}\,dx-\frac{1}{2}\sum_{Q\in\mathfrak{B}_{\Lambda}(\Omega)}\sum_{R\in\mathfrak{Y}(Q)}\int_{\partial Q\cap\partial R}^{*}\left(u_{Q}-u_{R}\right)v_{Q}\,(\mathbf{e}_{i}\cdot\mathbf{n}_{Q})\,dS.\label{eq:g2}
\end{equation}
In fact, if $u,\,v\in U_{\Lambda}^{1}$, then $u_{Q},\,v_{Q}\in V_{\Lambda}^{1}(\Omega)\,$
and hence, by $\eqref{eq:pinta-0},$ $\partial_{i}^{*}u_{Q}v_{Q}\in V_{\Lambda}^{0}(\Omega)$
and so
\[
\sqint\partial_{i}^{*}u_{Q}v_{Q}\theta_{Q}\,dx=\int_{Q}^{*}\partial_{i}^{*}u_{Q}v_{Q}\,dx.
\]
\end{rem}
\medskip{}

\begin{thm}
The operator $D_{i}:\,V_{\Lambda}(\Omega)\rightarrow V_{\Lambda}(\Omega)$,
given by Definition \ref{def:bellissima}, satisfies the requests
I, II, and III of Definition \ref{def:deri}. 
\end{thm}
\begin{proof}
Let us prove property I. If $u\theta_{Q},v\theta_{R}\in V_{\Lambda}(\Omega)$
and $\overline{Q}\cap\overline{R}=\emptyset$, by Definition \ref{def:bellissima},
\[
\sqint D_{i}\left(u\theta_{Q}\right)v\theta_{R}\,dx=0.
\]
 Set
\[
\delta:=\max\left\{ diam(Q)\ |\ Q\in\mathfrak{B}_{\Lambda}(\Omega)\right\} .
\]
If $q\in Q$ and $r\in R,$ then
\[
|q-r|>2\delta\Rightarrow\overline{Q}\cap\overline{R}=\emptyset,
\]
so, if $\sigma_{q}\in V_{\Lambda}^{0}(Q)$, and $r\in R,$ then
\[
|q-r|>2\delta\Rightarrow\overline{Q}\cap\overline{R}=\emptyset,
\]
and hence, if we set $\varepsilon_{0}>3\delta$,
\[
\bigcup\left\{ R\in\mathfrak{B}_{\Lambda}(\Omega)\ |\ \overline{Q}\cap\overline{R}\neq\emptyset\right\} \subset B_{\varepsilon_{0}}(q).
\]
Since $\sigma_{q}\in V_{\Lambda}^{0}(Q)$,
\[
\mathfrak{supp}\left(D_{i}\sigma_{q}\right)\subset\overline{\bigcup\left\{ R\in\mathfrak{B}_{\Lambda}(\Omega)\ |\ \overline{Q}\cap\overline{R}\neq\emptyset\right\} }\subset B_{\varepsilon_{0}}(q).
\]

We prove property II. If $u\in\left[C^{1}(\Omega)\right]^{*}\cap V_{\Lambda}(\Omega)$,
then $u=\sum_{_{Q\in\mathfrak{B}_{\Lambda}(\Omega)}}u\theta_{Q}$,
and hence $\forall x\in\partial Q\cap\partial R$, $u_{Q}(x)-u_{R}(x)=u(x)-u(x)=0.$
Then, by (\ref{eq:g1}), we have that, $\forall v\in V_{\Lambda}(\Omega)$ 

\begin{align*}
\sqint D_{i}uv\,dx & =\sum_{Q\in\mathfrak{B}_{\Lambda}(\Omega)}\sqint\partial_{i}^{*}uv_{Q}\theta_{Q}\,dx=\sqint\partial_{i}^{*}u\left(\sum_{Q\in\mathfrak{B}_{\Lambda}(\Omega)}v_{Q}\theta_{Q}\right)\,dx\\
 & =\sqint\partial_{i}^{*}uv\,dx.
\end{align*}
The conclusion follows from the arbitrariness of $v$.

Next let us prove property III. First we prove this property if $u,v\in U_{\Lambda}^{1}$.
By (\ref{eq:g2}), we have that
\begin{align}
\sqint D_{i}uv\,dx & =\sum_{Q\in\mathfrak{B}_{\Lambda}(\Omega)}\int_{Q}^{*}\partial_{i}^{*}uv\,dx-\frac{1}{2}\sum_{Q\in\mathfrak{B}_{\Lambda}(\Omega)}\sum_{R\in\mathfrak{Y}(Q)}\int_{\partial Q\cap\partial R}^{*}u_{Q}v_{Q}\,(\mathbf{e}_{i}\cdot\mathbf{n}_{Q})\,dS\nonumber \\
 & +\frac{1}{2}\sum_{Q\in\mathfrak{B}_{\Lambda}(\Omega)}\sum_{R\in\mathfrak{Y}(Q)}\int_{\partial Q\cap\partial R}^{*}u_{R}v_{Q}\,(\mathbf{e}_{i}\cdot\mathbf{n}_{Q})\,dS\label{eq:pilu}\\
 & =\sum_{Q\in\mathfrak{B}_{\Lambda}(\Omega)}\int_{Q}^{*}\partial_{i}^{*}uv\,dx-\frac{1}{2}\sum_{Q\in\mathfrak{B}_{\Lambda}(\Omega)}\int_{\partial Q}^{*}u_{Q}v_{Q}\,(\mathbf{e}_{i}\cdot\mathbf{n}_{Q})\,dS\\
 & +\frac{1}{2}\sum_{Q\in\mathfrak{B}_{\Lambda}(\Omega)}\sum_{R\in\mathfrak{Y}(Q)}\int_{\partial Q\cap\partial R}^{*}u_{R}v_{Q}\,(\mathbf{e}_{i}\cdot\mathbf{n}_{Q})\,dS.
\end{align}
Next we will compute $\sqint uD_{i}v\,dx$ and we will show that it
is equal to $-\sqint D_{i}uv\,dx$. So we replace $u$ with $v$,
in the above equality and we get
\begin{align}
\sum_{Q\in\mathfrak{B}_{\Lambda}(\Omega)}\int_{Q}^{*}uD_{i}v\,dx & =\sum_{Q\in\mathfrak{B}_{\Lambda}(\Omega)}\int_{Q}^{*}\partial_{i}^{*}v_{Q}u_{Q}\,dx-\frac{1}{2}\sum_{Q\in\mathfrak{B}_{\Lambda}(\Omega)}\int_{\partial Q}^{*}u_{Q}v_{Q}\,(\mathbf{e}_{i}\cdot\mathbf{n}_{Q})\,dS\nonumber \\
 & +\frac{1}{2}\sum_{Q\in\mathfrak{B}_{\Lambda}(\Omega)}\sum_{R\in\mathfrak{Y}(Q)}\int_{\partial Q\cap\partial R}^{*}v_{R}u_{Q}\,(\mathbf{e}_{i}\cdot\mathbf{n}_{Q})\,dS.\label{eq:jk}
\end{align}
Now, we compute $\sum_{_{Q\in\mathfrak{B}_{\Lambda}(\Omega)}}\smallint_{_{Q}}^{*}\partial_{i}^{*}v_{Q}u_{Q}\,dx$
and the last term of the above expression separately. We have that

\begin{align}
\sum_{Q\in\mathfrak{B}_{\Lambda}(\Omega)}\int_{Q}^{*}\partial_{i}^{*}v_{Q}u_{Q}dx & =-\sum_{Q\in\mathfrak{B}_{\Lambda}(\Omega)}\int_{Q}^{*}\partial_{i}^{*}u_{Q}v_{Q}dx+\sum_{Q\in\mathfrak{B}_{\Lambda}(\Omega)}\int_{\partial Q}^{*}u_{Q}v_{Q}\,(\mathbf{e}_{i}\cdot\mathbf{n}_{Q})\,dS.\label{eq:bimba}
\end{align}
Moreover, the last term in \eqref{eq:jk}, changing the order on which
the terms are added, becomes
\[
\sum_{Q\in\mathfrak{B}_{\Lambda}(\Omega)}\sum_{R\in\mathfrak{Y}(Q)}\int_{\partial Q\cap\partial R}^{*}v_{R}u_{Q}\,(\mathbf{e}_{i}\cdot\mathbf{n}_{Q})\,dS=\sum_{R\in\mathfrak{B}_{\Lambda}(\Omega)}\sum_{Q\in\mathfrak{Y}(R)}\int_{\partial Q\cap\partial R}^{*}v_{R}u_{Q}\,(\mathbf{e}_{i}\cdot\mathbf{n}_{Q})\,dS.
\]
In the right hand side we can change the name of the variables $Q$
and $R$:
\begin{align}
\sum_{Q\in\mathfrak{B}_{\Lambda}(\Omega)}\sum_{R\in\mathfrak{Y}(Q)}\int_{\partial Q\cap\partial R}^{*}v_{R}u_{Q}\,(\mathbf{e}_{i}\cdot\mathbf{n}_{Q})\,dS & =\sum_{Q\in\mathfrak{B}_{\Lambda}(\Omega)}\sum_{R\in\mathfrak{Y}(Q)}\int_{\partial Q\cap\partial R}^{*}v_{Q}u_{R}\,(\mathbf{e}_{i}\cdot\mathbf{n}_{R})\,dS\nonumber \\
 & =-\sum_{Q\in\mathfrak{B}_{\Lambda}(\Omega)}\sum_{R\in\mathfrak{Y}(Q)}\int_{\partial Q\cap\partial R}^{*}v_{Q}u_{R}\,(\mathbf{e}_{i}\cdot\mathbf{n}_{Q})\,dS.\label{eq:bella}
\end{align}
In the last step we have used the fact that $x\in\partial Q\cap\partial R\Rightarrow\mathbf{n}_{R}(x)=-\mathbf{n}_{Q}(x)$.
Replacing (\ref{eq:bimba}) and (\ref{eq:bella}) in (\ref{eq:jk})
we get 
\begin{align*}
\sum_{Q\in\mathfrak{B}_{\Lambda}(\Omega)}\int_{Q}^{*}uD_{i}v\,dx & =-\sum_{Q\in\mathfrak{B}_{\Lambda}(\Omega)}\int_{Q}^{*}\partial_{i}^{*}u_{Q}v_{Q}dx+\frac{1}{2}\sum_{Q\in\mathfrak{B}_{\Lambda}(\Omega)}\int_{\partial Q}^{*}u_{Q}v_{Q}\,(\mathbf{e}_{i}\cdot\mathbf{n}_{Q})\,dS\\
 & -\frac{1}{2}\sum_{Q\in\mathfrak{B}_{\Lambda}(\Omega)}\sum_{R\in\mathfrak{Y}(Q)}\int_{\partial Q\cap\partial R}^{*}v_{Q}u_{R}\,(\mathbf{e}_{i}\cdot\mathbf{n}_{Q})\,dS.
\end{align*}

Comparing (\ref{eq:pilu}) and the above equation, we get that

\begin{equation}
\forall u,v\in U_{\Lambda}^{1},\ \ \sqint D_{i}uv\,dx=-\sqint uD_{i}v\,dx.\label{eq:ppp}
\end{equation}

Let us prove property III in the general case. We have that 
\begin{equation}
D_{i}u=D_{i}u_{1}-\left(D_{i}P_{1}\right)^{\dagger}u_{0}\label{eq:bellissima-2-1}
\end{equation}
hence

\begin{align}
\sqint D_{i}uv\,dx & =\sqint D_{i}u_{1}v\,dx-\sqint\left(D_{i}P_{1}\right)^{\dagger}u_{0}v\,dx\nonumber \\
 & =\sqint D_{i}u_{1}v_{1}\,dx+\sqint D_{i}u_{1}v_{0}\,dx-\sqint u_{0}D_{i}P_{1}v\,dx\nonumber \\
 & =\sqint D_{i}u_{1}v_{1}\,dx+\sqint D_{i}u_{1}v_{0}\,dx-\sqint u_{0}D_{i}v_{1}\,dx.\label{eq:melinda}
\end{align}
Now, replacing $u$ with $v$ and applying property (\ref{eq:ppp})
for $u_{1},v_{1}\in U_{\Lambda}^{1}$, we get 
\begin{align*}
\sqint uD_{i}v\,dx & =\sqint u_{1}D_{i}v_{1}\,dx-\sqint D_{i}u_{1}v_{0}\,dx+\sqint u_{0}D_{i}v_{1}\,dx\\
 & =-\sqint D_{i}u_{1}v_{1}\,dx-\sqint D_{i}u_{1}v_{0}\,dx+\sqint u_{0}D_{i}v_{1}\,dx.
\end{align*}
Comparing the above equation with (\ref{eq:melinda}) we get that
\[
\sqint D_{i}uv\,dx=-\sqint uD_{i}v\,dx.
\]
\end{proof}
Before proving property IV we need the following lemma:
\begin{lem}
The following identity holds true: $\forall E\in\mathfrak{C}_{\Lambda}(\Omega)$,
$\forall u\in V_{\Lambda}(\Omega)\cap\left[C^{1}(\Omega)\right]^{*}$
and $\forall v\in V_{\Lambda}(\Omega)$, 
\begin{equation}
\sqint D_{i}\left(u\theta_{E}\right)v\,dx=\sqint\partial_{i}uv\theta_{E}\,dx-\int_{\partial E}^{*}uv\,(\mathbf{e}_{i}\cdot\mathbf{n}_{E})\,dS.\label{eq:cordelia}
\end{equation}
\end{lem}
\begin{proof}
We can write
\[
u\theta_{E}=\sum_{Q\in\mathfrak{B}_{\Lambda}(\Omega)}h_{Q}u\theta_{Q},
\]
where
\[
h_{Q}=\begin{cases}
1 & if\ Q\subset E\\
0 & if\ Q\nsubseteq E
\end{cases}
\]
Then, we have that 
\[
h_{Q}u-h_{R}u=u,
\]
 if and only if,
\[
R\in\mathfrak{Y}_{Q,E}^{+}:=\left\{ R\in\mathfrak{B}_{\Lambda}(\Omega)\cup\left\{ Q_{\infty}\right\} \,|\,\partial R\cap\partial E\neq\emptyset,\,Q\subset E,\,R\subset\Omega\setminus E\right\} .
\]
 Moreover, we have that,
\[
h_{Q}u-h_{R}u=-u,
\]
if and only if,
\[
R\in\mathfrak{Y}_{Q,E}^{-}:=\left\{ R\in\mathfrak{B}_{\Lambda}(\Omega)\cup\left\{ Q_{\infty}\right\} \,|\,\partial R\cap\partial E\neq\emptyset,\,Q\subset\Omega\setminus E,\,R\subset E\right\} .\,
\]
Otherwise, we have that 
\[
h_{Q}u-h_{R}u=0\,\,\,\,or\,\,\,\,\partial R\cap\partial E=\emptyset.
\]
Then, by Theorem \ref{thm:rigoletto},
\begin{align*}
\sqint D_{i}\left(u_{E}\theta_{E}\right)v\,dx & =\sum_{Q\in\mathfrak{B}_{\Lambda}(\Omega)}\sqint\partial_{i}^{*}\left(h_{Q}u\right)v\theta_{Q}\,dx-\frac{1}{2}\sum_{Q\in\mathfrak{B}_{\Lambda}(\Omega)}\sum_{R\in\mathfrak{Y}(Q)}\int_{\partial Q\cap\partial R}^{*}\left(h_{Q}u-h_{R}u\right)v\,(\mathbf{e}_{i}\cdot\mathbf{n}_{Q})\,dS\\
 & =\sum_{Q\in\mathfrak{B}_{\Lambda}(\Omega)}h_{Q}\sqint_{Q_{\Gamma}}\partial_{i}^{*}uvdx-\frac{1}{2}\sum_{Q\in\mathfrak{B}_{\Lambda}(\Omega)}\sum_{R\in\mathfrak{Y}_{Q,E}^{+}}\int_{\partial Q\cap\partial R}^{*}uv\,(\mathbf{e}_{i}\cdot\mathbf{n}_{Q})\,dS\\
 & +\frac{1}{2}\sum_{Q\in\mathfrak{B}_{\Lambda}(\Omega)}\sum_{R\in\mathfrak{Y}_{Q,E}^{-}}\int_{\partial Q\cap\partial R}^{*}uv\,(\mathbf{e}_{i}\cdot\mathbf{n}_{Q})\,dS\\
 & =\sqint_{E_{\Gamma}}\partial_{i}^{*}uv\,dx-\frac{1}{2}\sum_{Q\in\mathfrak{B}_{\Lambda}(\Omega),Q\subset E}\int_{\partial Q\cap\partial E}^{*}uv\,(\mathbf{e}_{i}\cdot\mathbf{n}_{E})\,dS\\
 & +\frac{1}{2}\sum_{Q\in\mathfrak{B}_{\Lambda}(\Omega),\,Q\subset\Omega\setminus E}\int_{\partial Q\cap\partial E}^{*}uv\,(\mathbf{e}_{i}\cdot(-\mathbf{n}_{E}))\,dS\\
 & =\sqint\partial_{i}^{*}uv\theta_{E}\,dx-\frac{1}{2}\int_{\partial E}^{*}uv\,(\mathbf{e}_{i}\cdot\mathbf{n}_{E})\,dS+\frac{1}{2}\int_{\partial E}^{*}uv\,(\mathbf{e}_{i}\cdot(-\mathbf{n}_{E}))\,dS\\
 & =\sqint\partial_{i}^{*}uv\theta_{E}\,dx-\int_{\partial E}^{*}uv\,(\mathbf{e}_{i}\cdot\mathbf{n}_{E})\,dS.
\end{align*}
 Then (\ref{eq:cordelia}) holds true.
\end{proof}
\begin{cor}
The operator $D_{i}:\,V_{\Lambda}(\Omega)\rightarrow V_{\Lambda}(\Omega)$
given by Definition \ref{def:bellissima} satisfies the request IV
of Definition \ref{def:deri}. 
\end{cor}
\begin{proof}
The result follows straighforward from (\ref{eq:cordelia}) just taking
$u=1$.
\end{proof}

\section{Some examples\label{sec:One-simple-application}}

We present a general minimization result and two very basic examples
which can be analyzed in the framework of ultrafunctions. We have
chosen these examples for their simplicity and also because we can
give explicit solutions. 

\subsection{A minimization result}

In this section we will consider a minimization problem. Let $\Omega$
be an open bounded set in $\mathbb{R}^{N}$ and let $\Xi\subset\partial\Omega$
be any nonempty portion of the boundary. We consider the following
problem: minimize 
\[
J(u)=\int_{\Omega}\left[\frac{1}{2}a(u)|\nabla u(x)|^{p}+f(x,u)\right]dx,\,\,p>1
\]
in the set
\[
\mathscr{C}^{1}(\Omega)\cap\mathscr{C}_{0}(\Omega\cup\Xi).
\]
We make the following assumptions:
\begin{enumerate}
\item $a(u)\geq0$ and $a(u)\geq k>0$ for $u$ sufficiently large;
\item $a(u)$ is lower semicontinuous;
\item $f(x,u)$ is a lower semicontinuous function in $u$, measurable in
$x$, such that $|f(x,u)|\leq M|u|^{q}$, with $0<q<p$ and $M\in\mathbb{R}^{+}$.
\end{enumerate}
Clearly, the above assumptions are not sufficient to guarantee the
existence of a solution, not even in a Sobolev space. We refer to
\cite{squa} for a survey of this problem in the framework of Sobolev
spaces. On the other hand, we have selected this problem since it
can be solved in the framework of the ultrafunctions. 

More exactly, this problem becomes: find an ultrafunction $u\in V_{\Lambda}(\Omega)$
which vanishes on $\Xi^{*}$ and minimizes
\[
J\text{\textdegree}(u):=\sqint_{\Omega}\left[\frac{1}{2}a^{*}(u)\,|Du(x)|^{p}-f^{*}(x,u)\right]dx,\,\,\,p>1.
\]

We have the following result:
\begin{thm}
\label{thm:verde}If assumptions 1,2,3 are satisfied, then the functional
$J\text{\textdegree}(u)$ has a minimizer in the space
\[
\left\{ v\in V_{\Lambda}(\Omega)\,|\,\forall x\in\Xi^{*},\,v(x)=0\,\right\} .
\]
Moreover, if $J(u)$ has a minimizer $w$ in $V_{\Lambda}(\Omega)$,
then $u=w\text{\textdegree}$.
\end{thm}
\begin{proof}
The proof is based on a standard approximation by finite dimensional
spaces. Let us observe that, for each finite dimensional space $V_{\lambda}$,
we can consider the approximate problem: find $u_{\lambda}\in V_{\lambda}$
such that

\[
J\text{\textdegree}(u_{\lambda})=\min_{v_{\lambda}\in V_{\lambda}}J\text{\textdegree}(v_{\lambda}).
\]
The above minimization problem has a solution, being the functional
coercive, due to the hypotheses on $a(\cdot)$ and the fact that $p>q$.
If we take a miminizing sequence $u_{\lambda}^{n}\in V_{\lambda}$,
then we can extract a subsequence weakly converging to some $u_{\lambda}\in V_{\lambda}.$
By observing that in finite dimensional spaces all norms are equivalent,
it follows also that $u_{\lambda}^{n}\to u_{\lambda}$ pointwise.
Then, by the lower-semicontinuity of $a$ and $f$, it follows that
the pointwise limit satisfies

\[
J\text{\textdegree}(u_{\lambda})\leq\liminf J\text{\textdegree}(u_{\lambda}^{n}).
\]
Next, we use the very general properties of $\Lambda$-limits, as
introduced in Section \ref{subsec:Lambda-limit}. We set
\[
u:=\lim_{\lambda\uparrow\Lambda}u_{\lambda}.
\]
Then, taking a generic $v:=\lim_{\lambda\uparrow\Lambda}v_{\lambda}$,
from the inequality $J\text{\textdegree}(u_{\lambda})\leq J\text{\textdegree}(v_{\lambda})$,
we get
\[
J\text{\textdegree}(u)\leq J\text{\textdegree}(v)\qquad\forall v\in V_{\Lambda}(\Omega).
\]
The last statement is trivial.
\end{proof}
Clearly, under this generality, the solution $u$ could be very wild;
however, we can state a regularization result which allows the comparison
with variational and classical solutions.
\begin{thm}
\label{thm:rosso}Let the assumptions of the above theorem hold. If
\[
\mathcal{H}^{N-1}(\Xi)>0
\]
and there exists $\nu\in\mathbb{R}$ such that
\[
a(u)\geq\nu>0,\,\,\,
\]
then, the minimizer has the following form
\[
u=w^{\circ}+\psi,
\]
where $w\in H^{1,p}(\Omega)$ and $\psi$ is null in the sense of
distributions, namely
\[
\forall\varphi\in\mathscr{D(\mathrm{\Omega})},\,\,\sqint\psi\varphi^{*}dx\sim0.
\]
In this case 
\[
J\text{\textdegree}(u)\sim\underset{v\in V(\mathrm{\Omega})}{inf}J\text{\textdegree}(v)
\]
with $V(\Omega)$ as in Definition \ref{def:AB}. Moreover, if in
addition $a(u)<M$, with $\,M\in\mathbb{R}$, we have that 
\[
\left\Vert \psi\right\Vert _{H^{1,p}(\Omega)}\sim0
\]
and $J\text{\textdegree}(u)\sim J(w)$. Finally, if \textup{$J(u)$
}\textup{\emph{has a minimizer in }}\textup{$w\in H^{1,p}(\Omega)\cap\mathscr{C}(\Omega)$,
}\textup{\emph{then }}$u=w^{\circ}$ and $J\text{\textdegree}(u)=J(w)$. 
\end{thm}
\begin{proof}
Under the above hypotheses the minimization problem has an additional
a priori estimates in $H^{1,p}(\Omega),$ due to the fact that $a(\cdot)$
is bounded away from zero. Moreover, the fact that the function vanishes
on a non trascurable $(N-1)$-dimensional part of the boundary, shows
that the generalized Poincar� inequality holds true. Hence, by Proposition
\ref{prop:carola}, the approximating net $\left\{ u_{\lambda}\right\} $
has a subnet $\left\{ u_{n}\right\} $ such that
\[
u_{n}\rightarrow u\quad\text{weakly in}\quad H^{1,p}(\Omega).
\]
This proves the first statement, since obviously, $\psi:=u-w^{\circ}$
vanishes in the sense of distributions. In this case, in general the
minimum is not achieved in $V(\mathrm{\Omega})$ and hence $J\text{\textdegree}(w^{\circ}+\psi)<J(w)$.

Next, if $a(\cdot)$ is bounded also from above, by classical results
of semicontinuity of De Giorgi (see Boccardo \cite{boccardo} Section
9, Thm. 9.3) $J$ is weakly l.s.c. Thus $u$ is a minimizer and, by
well known results, $u_{n}\rightarrow u$ strongly in $H^{1,p}(\Omega).$
This implies, by Proposition \ref{prop:carola}, that $u\sim w^{\circ},$
and hence $\psi$ is infinitesimal in $H^{1,p}(\Omega),$ proving
the second part. 

Finally, if the minimizer is a function $w\in H^{1,p}(\Omega)\cap\mathscr{C}(\Omega)\subset V_{\Lambda}(\Omega)$,
we have that $u_{\lambda}=w\text{\textdegree}$ eventually; then 

\[
\left\Vert \psi\right\Vert _{H^{1,p}(\Omega)}=0.
\]
\end{proof}

\subsection{The Poisson problem in $\mathbb{R}^{2}$}

Now we cosider this very classical problem:
\begin{equation}
-\triangle u=\varphi(x),\,\,\,\,\varphi\in\mathscr{D}(\mathbb{R}^{N}).\label{eq:poisson}
\end{equation}

If $N\geq3$, the solution is given by 
\[
\varphi(x)\ast\frac{|x|^{-N+2}}{(N-2)\omega_{N}}
\]
and it can be characterized in several ways. 

First of all, it is the only solution the Schwartz space $\mathscr{S'}$
of tempered distributions obtained via the equation 
\begin{equation}
\widehat{u}(\xi)=\frac{\widehat{\varphi}(\xi)}{|\xi|^{2}}\label{eq:maria}
\end{equation}
where $\widehat{T}$ denotes the Fourier tranform of $T$.

Moreover, it is the minimizer of the Dirichlet integral
\[
J(u)=\int\left[\frac{1}{2}|\nabla u(x)|^{2}-\varphi(x)u(x)\right]dx
\]
in the space $\mathscr{D}^{1,2}(\mathbb{R}^{N})$ which is defined
as the completion of $\mathcal{C}^{1}\left(\mathbb{R}^{N}\right)$
with respect to the Dirichlet norm
\[
\left\Vert u\right\Vert =\sqrt{\int|\nabla u(x)|^{2}dx}.
\]

Each of these characterizations provides a different method to prove
its existence.

The situation is completely different when $N=2.$ In this case, it
is well known that the fundamental class of solutions is given by
\[
2\pi\cdot\varphi(x)\ast log|x|,
\]
 however none of the previus characterization makes sense. In fact,
we cannot use equation (\ref{eq:maria}), since $\frac{1}{|\xi|^{2}}\notin L_{loc}^{1}(\mathbb{R}^{2})$
and hence $\frac{1}{|\xi|^{2}}$ does not define a tempered distribution.
Also, the space $\mathscr{D}^{1,2}(\mathbb{R}^{2})$ is not an Hilbert
space and the functional $J(u)$ is not bounded from below in $\mathscr{D}^{1,2}(\mathbb{R}^{2})$. 

On the contrary, using the theory of ultrafunctions, we can treat
equation (\ref{eq:poisson}) independently of the dimension. 

First of all, we recall that in equation (\ref{eq:poisson}) with
$N\geq3$, the boundary conditions are replaced by the condition $u\in\mathscr{D}^{1,2}(\mathbb{R}^{N})$.
This is a sort of Dirichlet boundary condition. In the theory of ultrafunctions
it is not necessary to replace the Dirichlet boundary condition with
such a trick. In fact we can reformulate the problem in the following
way: find $u\in V_{\Lambda}(B_{R})$ such that 
\begin{align*}
-\triangle\text{\textdegree}u & =\varphi\text{\textdegree}(x)\,\,\,in\,\,B_{R}\\
u & =0\,\,\,\,\,\,\,\,\,\,\,on\,\,\partial B_{R}
\end{align*}
where $\triangle\text{\textdegree}$ is the ``generalized'' Laplacian
defined in Section \ref{subsec:gelsomina} and $R$ is an infinite
number such that $\chi_{_{B_{R}}}\in V_{\Lambda}(\mathbb{R}^{N})$.
\footnote{Such an $R$ exists by overspilling (see e.g. \cite{rob,keisler76,nelson});
in fact for any $r\in\mathbb{R}$, $\chi_{_{B_{r}}}\in V_{\Lambda}(\mathbb{R}^{N})$. }

Clearly, the solutions of the above problem are the minimizers of
the Dirichlet integral
\[
J\text{\textdegree}(u)=\sqint\left[\frac{1}{2}|Du(x)|^{2}-\varphi\text{\textdegree}(x)u(x)\right]dx
\]
in the space $u\in V_{\Lambda}(B_{R})$, with the Dirichlet boundary
condition. Notice that, in the case of ultrafunctions, the problem
has the same structure independently of $N$. In order to prove the
existence, we can use Theorem \ref{thm:verde}\footnote{The fact that $\Omega$ is a standard set while $B_{R}$ is an internal
set does not change the proof.}. The fact that $J\text{\textdegree}(u)$ may assume infinite values
does not change the structure of the problem and shows the utility
of the use of infinite quantities. The relation between the classical
solution $w$ and the ultrafunction $u$ is given by 
\[
u=w\text{\textdegree}+\psi
\]
with
\[
St_{\mathscr{D}'}\psi=0.
\]

Some people might be disappointed that $u$ depends on $R$ and it
is not a standard function; if this is the case it is sufficient to
take
\[
w=St_{\mathscr{D}'}u
\]
and call $w$ \emph{the standard solution of the Poisson problem with
Dirichlet boundary condition at $\infty$.} In this way we get the
usual fundamental class of solutions and they can be characterized
in the usual way also in the case $N=2$. Concluding, in the framework
of ultrafunctions, the Poisson problem with Dirichlet boundary condition
is the \emph{\uline{same}} problem independently of the space dimension
and and it is very similar to the same problem when $R$ is finite. 

This fact proves that the use of infinite numbers is an advantage
which people should not ignore.

\subsection{An explicit example}

If the assumptions of Theorem \ref{thm:rosso} do not hold true, the
solution could not be related to any standard object. For example,
if $\mathcal{H}^{N-1}(\Xi)=0$ and $f(x,u)>k|u|^{s},\,\,(p<N,\,\,k>0,\,\,0<s<q)$,
the generalized solution $u(x)$ takes infinite values for every $x\in\Omega$.
However, there are cases in which $u(x)$ can be identified with a
standard and meaningful function, but the minimization problem makes
no sense in the usual mathematics. In the example which we will present
here, we deal with a functional which might very well represent a
physical model, even if the explicit solution cannot be interpreted
in a standard world, since it involves the square of a measure (namely
$\delta{}^{2}$).

Let us consider, for $\gamma>0$, the one dimensional variational
problem of finding the minimum of the functional 
\begin{equation}
J(u)=\int_{0}^{1}\frac{1}{2}a(u)|u'(x)|^{2}-\gamma u(x)\,dx\label{eq:gei}
\end{equation}
among the functions such that $u(0)=0$. In particular we are interested
in the case in which $a$ is the following degenerate function
\[
a(s)=\left\{ \begin{aligned} & 1\quad &  & \text{ if }s\in\left(-\infty,1\right)\cup\left(2,+\infty\right),\\
 & 0\quad &  & \text{ if }s\in\left[1,2\right].
\end{aligned}
\right.
\]
Formally, the Euler equation, if $u\notin\left[1,2\right]$, is 
\[
u''(x)=-\gamma.
\]
We recall that, by standard arguments, 
\[
u(1)\neq1\Rightarrow u'(1)=0.
\]
Hence, if $\gamma<2$, the solution is explicitly computed 
\[
u(x)=\frac{\gamma}{2}(2x-x^{2}),
\]
since it turns out that $0\leq u(x)<1$ for all $x\in(0,1)$ and then
the degeneracy does not take place.

If $\gamma>2$, we see that the solution does not live in $H^{1}(0,1)$,
hence the problem has not a ``classical'' weak solution. More exactly
we have the following result:~
\begin{thm}
If \textup{$\gamma>2$ }\textup{\emph{then the functional (\ref{eq:gei})
has a unique minimizer given by }}
\[
u(x)=\left\{ \begin{aligned} & \frac{1}{2}\left(2\gamma x-\gamma x^{2}\right)\quad & 0<x<\xi\\
 & \frac{1}{2}\left(-\gamma x^{2}+2\gamma x+2\right)\quad & \xi<x<1
\end{aligned}
\right.
\]
where \textup{$\xi\in(0,1)$ }\textup{\emph{is a suitable real number
which depends on $\gamma$ (see }}Figure 2).
\end{thm}
\begin{proof}
First, we show that the generalized solution has at most one discontinuity.
In fact, for $\gamma>2$ the solution satisfies $u(\xi)=1$, for some
$0<\xi<1$, and at that point the classical Euler equations are not
anymore valid. On the other hand, where $u>2$, the solution satisfies
a regular problem, hence we are in the situation of having at least
the following possible candidate as solution with a jump at $\xi=\frac{\gamma-\sqrt{\gamma^{2}-2\gamma}}{\gamma}=1-\sqrt{1-\frac{2}{\gamma}}$~
and a discontinuity of derivatives at some $\xi<\eta<1$. 
\begin{figure}[h]
\centering \includegraphics[width=7cm]{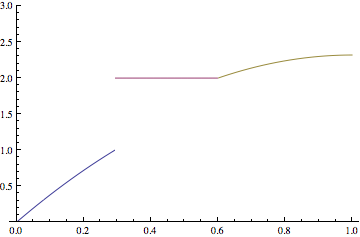}\label{fig:1}\caption{The function ${u}(x)$ for $\gamma=4$}
\end{figure}
In the specific case, we have (see Figure 1)
\[
u(x)=\left\{ \begin{aligned} & \frac{1}{2}\left(2\gamma x-\gamma x^{2}\right)\quad & 0<x<\xi\\
 & 2\quad & \xi<x<\eta\\
 & \frac{\gamma\eta^{2}}{2}-\gamma\eta-\frac{\gamma x^{2}}{2}+\gamma x+2\quad & \eta<x<1.
\end{aligned}
\right.
\]
We now show that this is not possible because the functional takes
a lower value on the solution with only a jump at $x=\xi$. In fact,
if we consider the function $\tilde{u}(x)$ defined as follows
\[
\tilde{u}(x)=\left\{ \begin{aligned} & \frac{1}{2}\left(2\gamma x-\gamma x^{2}\right)\quad & 0<x<\xi\\
 & \frac{1}{2}\left(-\gamma x^{2}+2\gamma x+2\right)\quad & \xi<x<1
\end{aligned}
\right.
\]
we observe that $u=\tilde{u}$ in $[0,\xi]$, while $u'=\tilde{u}'=\gamma(1-x)$
for all $x\notin[\xi,\eta]$ and, by explicit computations, we have
\[
J(\tilde{u})-J(u)=\gamma^{2}\left[-\frac{\eta}{2}+\frac{\eta^{2}}{2}-\frac{\eta^{3}}{6}+\frac{\xi}{2}-\frac{\xi^{2}}{2}+\frac{\xi^{3}}{6}\right]=\gamma^{2}(\Phi(\eta)-\Phi(\xi))<0,\quad\xi<\eta
\]
where $\Phi(s)=-\frac{s}{2}+\frac{s^{2}}{2}-\frac{s^{3}}{6}$ is strictly
decreasing, since $\Phi'(s)=-\frac{1}{2}(s-1)^{2}\leq0$. 
\begin{figure}[h]
\centering \includegraphics[width=7cm]{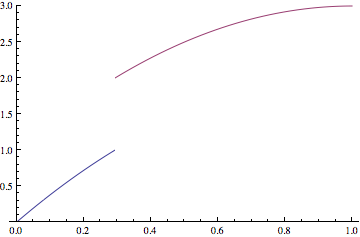} \label{fig:2}\caption{The function $\tilde{u}(x)$ for $\gamma=4$}
\end{figure}

Actually, the solution is the one shown in Figure 2. Next we show
that there exists a unique point $\xi$ such that the minimum is attained.
We write the functional $J(u)$, on a generic solution with a single
jump from the value $u=1$ to the value $u=2$ at the point $0<\xi<1$
and such that the Euler equation is satisfied before and after $\xi$.
We obtain the following value for the functional (in terms of the
point $\xi$) 
\[
J(u)=F(\xi)=\frac{\gamma^{2}\xi^{3}}{8}-\frac{\gamma^{2}\xi^{2}}{2}+\frac{\gamma^{2}\xi}{2}-\frac{\gamma^{2}}{6}+\frac{3\gamma\xi}{2}-2\gamma+\frac{1}{2\xi}.
\]
We observe that, $\forall\,\gamma>2$ 
\[
F(0^{+})=+\infty\qquad\text{and}\qquad F(1)=-\frac{\gamma^{2}}{24}-\frac{\gamma}{2}+\frac{1}{2}<0.
\]
To study the behavior of $F(\xi)$ one has to solve some fourth order
equations (this could be possible in an explicit but cumbersome way),
so we prefer to make a qualitative study. We evaluate 
\[
\begin{aligned}F'(\xi) & =\frac{3\gamma^{2}\xi^{2}}{8}-\gamma^{2}\xi+\frac{\gamma^{2}}{2}+\frac{3\gamma}{2}-\frac{1}{2\xi^{2}},\\
F''(\xi) & =\frac{3\gamma^{2}\xi}{4}-\gamma^{2}+\frac{1}{\xi^{3}},\\
F'''(\xi) & =\frac{3\gamma^{2}}{4}-\frac{3}{\xi^{4}},
\end{aligned}
\]
hence we have that 
\[
F'''(\xi)<0\quad\text{ if and only if }\quad0<\xi<\sqrt{\frac{2}{\gamma}}<1.
\]
Consequently the function $F''(\xi)$, which nevertheless satisfies
$\forall\,\gamma>2$ 
\[
F''(0^{+})=+\infty,\qquad F''(1)=1-\frac{\gamma^{2}}{4}<0,
\]
has a unique negative minimum at the point $\sqrt{\frac{2}{\gamma}}$.
\begin{figure}[h]
\centering \includegraphics[width=7cm]{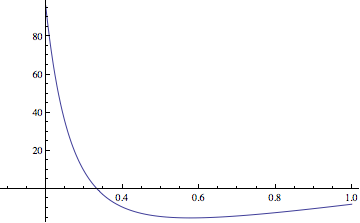}\caption{$F''(\xi)$ for $\gamma=4$}
\end{figure}

From this, we deduce that there exists one and only one point $0<\xi_{0}<\sqrt{\frac{2}{\gamma}}$
such that 
\[
\begin{aligned} & F''(\xi)>0 & \qquad0<\xi<\xi_{0}\\
 & F''(\xi)<0 & \qquad\xi_{0}<\xi\leq1.
\end{aligned}
\]
From the sign of $F''$ we get that $F'$ is strictly increasing in
$(0,\xi_{0})$ and decreasing in $(\xi_{0},1)$. Next $\forall\,\gamma>2$
\[
F'(0^{+})=-\infty,\qquad F'(1)=-\frac{\gamma^{2}}{8}+\frac{3\gamma}{2}-\frac{1}{2}\qquad
\]
hence, in the case that 
\[
-\frac{\gamma^{2}}{8}+\frac{3\gamma}{2}-\frac{1}{2}>0\quad\text{ that is }\gamma<2\left(3+2\sqrt{2}\right)\sim11.656...,
\]
then $F'$ has a single zero $\xi_{1}\in(0,\xi_{0})$ and, being a
change of sign, $\xi_{1}$ is a point of absolute minimum for $F(\xi)$.

If $\gamma\geq2\left(3+2\sqrt{2}\right)$ the above argument fails.

\begin{figure}[h]
\centering \includegraphics[width=7cm]{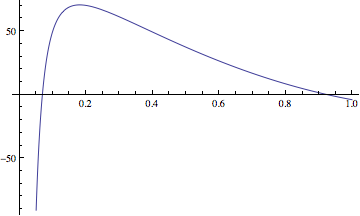}\caption{$F'(\xi)$ for $\gamma=14$}
\end{figure}

In this case we can observe that 
\[
F'(1/\gamma)=\frac{\gamma}{2}+\frac{3}{8}>0,
\]
hence $F'(1)$, which is negative at $\xi=1$ and near $\xi=0$ vanishes
exactly two times, at the point $\xi_{1}$, which is a point of local
minimum and at another point $\xi_{2}>\xi_{1}$, which is a point
of local maximum. Hence, to find the absolute minimum, we have to
compare the value of $F(\xi_{1})$ with that of $F(1)$.

In particular, we have that $\xi_{1}<\sqrt{2/\gamma}$ hence, we can
show that the minimum is not at $\xi=1$ simply by observing that
we can find at least a point where $F(\xi)<F(1)$ and this point is
$\sqrt{2/\gamma}$. In fact 
\[
M(\gamma):=F(\sqrt{2/\gamma})-F(1)=\frac{\gamma^{3/2}}{\sqrt{2}}-\frac{\gamma^{2}}{8}-\frac{5\gamma}{2}+2\sqrt{2}\sqrt{\gamma}-\frac{1}{2}\leq0.
\]
In particular $M(2)=0$ and 
\[
M'(\gamma)=\frac{1}{4}\left(-\gamma+3\sqrt{2}\sqrt{\gamma}+\frac{4\sqrt{2}}{\sqrt{\gamma}}-10\right)<0.
\]
This follows since by the substituting $\sqrt{\gamma}\to\chi$, we
have to control the sign of the cubic 
\[
\tilde{M}(\chi)=-\chi^{3}+3\sqrt{2}\chi^{2}-10\chi+4\sqrt{2}
\]
which is negative for all $\chi\geq1$, since $\tilde{M}(1)=-11+7\sqrt{2}$,
while 
\[
\tilde{M}'(\chi)=-3\chi^{2}+6\sqrt{2}\chi-10,
\]
is a parabola with negative minimum.
\end{proof}
\begin{rem}
It could be interesting to study this problem in dimension bigger
that one, namely, to minimize
\begin{equation}
J\text{\textdegree}(u)=\sqint_{\Omega}\left(\frac{1}{2}a(u)|Du(x)|^{2}-\gamma u(x)\right)\,dx\label{eq:gei-1}
\end{equation}
in the set 
\[
\left\{ v\in V_{\Lambda}(\Omega)\,|\,\forall x\in\Xi^{*},\,u(x)=0\,\right\} 
\]
and in particular to investigate the structure of the singular set
of $u$, both in the general case and in some particular situations
in which it is possible to find explicit solutions (e.g. $\Omega=B_{R}(0)$).
\end{rem}


\begin{thebibliography}{10}
\bibitem{benci99} Benci V., \textsl{An algebraic approach to nonstandard
analysis, }in: Calculus of Variations and Partial differential equations,
(G.Buttazzo, et al., eds.), Springer, Berlin (1999), 285-326.

\bibitem{ultra} Benci V., \textsl{Ultrafunctions and generalized
solutions,} Adv. Nonlinear Stud. 13, (2013), 461\textendash 486, arXiv:1206.2257.

\bibitem{BGG} V. Benci, S. Galatolo, M. Ghimenti, \emph{An elementary
approach to Stochastic Differential Equations using the infinitesimals},
in Contemporary Mathematics 530, Ultrafilters across Mathematics,
American Mathematical Society, (2010), p. 1-22.

\bibitem{BDN2003} Benci V., Di Nasso M., \textsl{Alpha-theory: an
elementary axiomatic for nonstandard analysis}, Expo. Math. 21, (2003),
355-386.

\bibitem{belu2012} Benci V., Luperi Baglini L., \textsl{A model problem
for ultrafunctions}, in: Variational and Topological Methods: Theory,
Applications, Numerical Simulations, and Open Problems. Electron.
J. Diff. Eqns., Conference 21 (2014), 11-21.

\bibitem{belu2013} Benci V., Luperi Baglini L., \textsl{Basic Properties
of ultrafunctions,} in: Analysis and Topology in Nonlinear Dierential
Equations (D. G. Figuereido, J. M. do O, C. Tomei eds.), Progress
in Nonlinear Dierential Equations and their Applications, 85 (2014),
61-86.

\bibitem{milano} Benci V., Luperi Baglini L., \textsl{Ultrafunctions
and applications}, DCDS-S, Vol. 7, No. 4, (2014), 593-616. arXiv:1405.4152.

\bibitem{algebra} Benci V., Luperi Baglini L., \textsl{A non archimedean
algebra and the Schwartz impossibility theorem,}, Monatsh. Math. (2014),
503-520.

\bibitem{beyond} Benci V., Luperi Baglini L., \textsl{Generalized
functions beyond distributions}, AJOM 4, (2014), arXiv:1401.5270.

\bibitem{gauss} Benci V., Luperi Baglini L., \emph{A generalization
of Gauss' divergence theorem}, in: Recent Advances in Partial Dierential
Equations and Applications, Proceedings of the International Conference
on Recent Advances in PDEs and Applications (V. D. Radulescu, A. Sequeira,
V. A. Solonnikov eds.), Contemporary Mathematics (2016), 69-84.

\bibitem{nap} V. Benci, L. Horsten, S. Wenmackers - \emph{Non-Archimedean
probability}, Milan J. Math. 81 (2013), 121-151. arXiv:1106.1524.

\bibitem{boccardo} Boccardo L., Croce, G. - \emph{Elliptic partial
differential equations}, De Gruyter, (2013)

\bibitem{col85}Colombeau, J.-F. \emph{Elementary introduction to
new generalized functions}. North-Holland Mathematics Studies, 113.
Notes on Pure Mathematics, 103. North-Holland Publishing Co., Amsterdam,
1985

\bibitem{eva-gar}Evans L.C., Gariepy R.F. - \emph{Measure theory
and fine properties of functions - }Studies in Advanced Mathematics.
CRC Press, Boca Raton, FL, 1992.

\bibitem{keisler76} Keisler H.~J., \emph{F}\textsl{oundations of
Infinitesimal Calculus}, Prindle, Weber \& Schmidt, Boston, (1976).

\bibitem{nelson}Nelson, E. - \emph{Internal Set Theory: A new approach
to nonstandard analysis}, Bull. Amer. Math. Soc., 83 (1977), 1165\textendash 1198.

\bibitem{rob} Robinson A., N\textsl{on-standard Analysis,}Proceedings
of the Royal Academy of Sciences, Amsterdam (Series A) 64, (1961),
432-440.

\bibitem{sa59}Sato, M. \emph{, Theory of hyperfunctions.} II. J.
Fac. Sci. Univ. Tokyo Sect. I 8 (1959) 139-193.

\bibitem{sa60}Sato, M., \emph{Theory of hyperfunctions}. II. J. Fac.
Sci. Univ. Tokyo Sect. I 8 (1960) 387\textendash 437.

\bibitem{squa}Squassina M., \emph{Exstence, multiplicity, perturbation
and concentration results for a class of quasi linear elliptic problems,}
Electronic Journal of Differential Equations, Monograph 07, 2006,
(213 pages). ISSN: 1072-6691. URL: http://ejde.math.txstate.edu or
http://ejde.math.unt.edu ftp ejde.math.txstate.edu 
\end{thebibliography}
\end{document}